\documentclass[letterpaper, 10 pt, conference]{ieeeconf}  

\IEEEoverridecommandlockouts                              

\usepackage{graphics} 
\usepackage{epsfig} 
\usepackage{mathptmx} 
\usepackage{times} 
\usepackage{amsmath} 
\usepackage{amssymb}  
\usepackage{algorithm}
\usepackage{algorithmic}
\usepackage{tabu}
\usepackage{multirow}
\usepackage{caption}
\usepackage{rotating}
\usepackage{graphicx}
\usepackage{color}
\newcommand{\fy}[1]{{\color{black}#1}}

\newtheorem{assumption}{Assumption}
\newtheorem{remark}{Remark}
\newtheorem{definition}{Definition}
\newtheorem{lemma}{Lemma}
\newtheorem{proposition}{Proposition}
\newtheorem{theorem}{Theorem}

\title{\LARGE \bf
An Iterative Regularized Incremental Projected Subgradient Method for a Class of Bilevel Optimization Problems
}

\author{Mostafa Amini$^{1}$ and Farzad Yousefian$^{2}$
\thanks{$^{1}$Mostafa Amini and $^{2}$Farzad Yousefian are with the Department of Industrial Engineering and Management, Oklahoma State University, Stillwater, OK 74078, USA,
        {\tt\small moamini@okstate.edu, farzad.yousefian@okstate.edu}}%
}

\begin{document}

\maketitle
\thispagestyle{empty}
\pagestyle{empty}

\begin{abstract}

We study a class of bilevel convex optimization problems where the goal is to find the minimizer of \fy{an objective function} in the upper level, among the set of all optimal solutions of \fy{an optimization} problem in the lower level. A wide range of problems in convex optimization can be formulated \fy{using} this class. \fy{An important example is the case where an optimization problem is ill-posed.} \fy{In this paper, our interest lies in addressing the bilevel problems, where the lower level objective is given as a finite sum of separate nondifferentiable convex component functions. This is the case} in a variety of applications in distributed optimization, such as large-scale data processing in machine learning and neural networks. \fy{To the best of our knowledge, this class of bilevel problems, with a finite sum in the lower level, has not been addressed before. Motivated by this gap, we} develop an iterative regularized incremental subgradient method\fy{,} \fy{where the agents update their iterates in a cyclic manner using a regularized subgradient}. Under a suitable choice of \fy{the} regularization \fy{parameter sequence}, we establish the convergence of the proposed algorithm and \fy{derive a} rate of $\mathcal{\cal O} \left({1}/k^{0.5-\epsilon}\right)$ \fy{in terms of} the lower level objective function for an arbitrary small $\epsilon>0$. \fy{We} present the performance of the algorithm on a binary text classification problem.

\end{abstract}

\section{Introduction}

In this paper, we consider a class of bilevel optimization problems as follows
\begin{align}\label{def:SL-INC} \tag{$P_f^h$}
\displaystyle \mbox{minimize}& \qquad h(x)\\
\mbox{subject to} & \qquad x \in X^* \triangleq \arg\min_{y\in X}f(y) \notag.
\end{align}
where $f,h :\mathbb{R}^n \rightarrow \mathbb{R}$ \fy{denote the lower and upper level objective functions, respectively, and $X \subseteq \mathbb{R}^n $ is a constraint set}.
This is called the \textit{selection problem} (\cite{Tseng07,Sabach17}) as we are selecting among optimal solutions of a lower level problem, one that minimizes the objective function $h$.
\fy{In particular}, we \fy{consider the case where the lower level objective function is} given as $f(x) \triangleq \sum_{i=1}^{m} f_i(x)$, where $f_i:\mathbb{R}^n \rightarrow \mathbb{R}$ is the $i$th component function for $i=1,\cdots,m$.

We make the following basic assumptions.

\begin{assumption}[{\bf Problem properties}]\label{assum:properties}\ 

	\begin{itemize}
		\item[(a)] The set $X \subset \mathbb{R}^n$ is nonempty, compact and convex; also $X\subseteq int(dom(f) \cap dom(h))$.
		\fy{\item[(b)] The functions $f_i(x)$ for $i=1,\cdots, m$ are proper, closed, convex, and possibly nondifferentiable.
		\item[(c)] The function $h$ is strongly convex with parameter $\mu_h>0$ and possibly nondifferentiable.}
	\end{itemize}
\end{assumption}

\fy{Next}, we \fy{present} two instances of the applications of formulation \eqref{def:SL-INC}.
\subsection{Example problems}
\noindent {\bf(i) Constrained nonlinear optimization:} \fy{Consider a constrained} convex optimization problem \fy{given} as
\begin{align*}
\displaystyle \mbox{minimize}& \qquad h(x)\\
\mbox{subject to} &\qquad q_i(x) \leq 0, \ \hbox{for } i=1,\cdots,m\\
&\qquad x\in X
\end{align*}
where $X \subseteq \mathbb{R}^n$ \fy{is an easy-to-project constraint set}, $h,q_i :\mathbb{R}^n \rightarrow \mathbb{R} $ for all $i=1,\cdots,m$ \fy{are convex (and possibly nonlinear) funcitons}. This problem can be reformulated as \eqref{def:SL-INC} by setting (cf.  \cite{Solodov072}) 
\begin{align*}f(x) \triangleq \sum_{i=1}^{m} f_i(x)= \sum_{i=1}^{m} \max \{0, q_i(x)\}.\end{align*} 

\noindent {\bf (ii) Ill-posed distributed optimization:} 
An \fy{optimization} problem is called ill-posed when it has multiple optimal solutions or it is very sensitive to data perturbations \cite{Tikhonov77}. For instance, in applications arising in machine learning, consider the \textit{empirical risk minimization problem} where the goal is to minimize the total loss $\sum_{i=1}^{m} \mathfrak{L} (a_ix,b_i)$, where  $a_i$ is the input, $b_i$ is the output of $i$th observed datum and $\mathfrak{L}$ is the loss function. For example, in \textit{logisitic loss regression}, $\mathfrak{L}$ is merely convex. In these cases, another criterion such as sparsity may be taken into account for the optimal solution. So, to induce sparsity, a secondary objective function $h$ is considered in the given problem. For instance, the well-known \textit{elastic net} regularization 
can be used as function $h$. Hence, to address ill-posedness, the following bilevel optimization model is considered \cite{Tseng07,Sabach17}:
\begin{align*}
\displaystyle \mbox{minimize}& \qquad \|x\|_1 + \mu \|x\|_2^2\\
\mbox{subject to} &\qquad x \in \arg \min_{y \in X} \sum_{i=1}^{m} \mathfrak{L} (a_i^T y,b_i),
\end{align*} \\
where $\mu>0$ regulates the trade-off between $\ell_1$ and $\ell_2$ norms.

\begin{table*}[t]
	
	\centering
	\captionof{table}{\scriptsize{Comparison \fy{of Methods for Solving the} Bilevel Optimization Problem \fy{\eqref{def:SL-INC}}.}}
	\scalebox{1.2}{
\begin{tabular}{c|c|c|c|c|c}
	Paper                                                                        & \begin{tabular}[c]{@{}c@{}}Assumptions on \\ lower level obj. function\end{tabular}                                                    & \begin{tabular}[c]{@{}c@{}}Assumptions on \\ upper level obj. function\end{tabular}                         & Methodology                                                                         & Metric      & Convergence                                               \\ \hline \hline
	\multirow{2}{*}{\cite{Solodov07}}                           & \multirow{2}{*}{convex, locally Lipschitz}                                                                                         & \multirow{2}{*}{convex, smooth}                                                                         & \multirow{2}{*}{iter. regu.}                                                        & $h_k-h^*$  & \multirow{2}{*}{asymptotic}                               \\ \cline{5-5}
	&                                                                                                                                    &                                                                                                         &                                                                                     & $f_k-f^*$ &                                                           \\ \hline
	\multirow{2}{*}{\cite{Solodov072}}                           & \multirow{2}{*}{convex, nonsmooth}                                                                                                            & \multirow{2}{*}{convex, nonsmooth}                                                                                 & \multirow{2}{*}{iter. regu.}                                                        & $h_k-h^*$  & \multirow{2}{*}{asymptotic}                               \\ \cline{5-5}
	&                                                                                                                                    &                                                                                                         &                                                                                     & $f_k-f^*$ &                                                           \\ \hline
	\multirow{2}{*}{\cite{Beck14}}                              & \multirow{2}{*}{convex, Lipschitz}                                                                                                 & \multirow{2}{*}{\begin{tabular}[c]{@{}c@{}}strongly convex, continuously\\ differentiable\end{tabular}} & \multirow{2}{*}{MNG}                                                                & $h_k-h^*$  & asymptotic                                                \\ \cline{5-6} 
	&                                                                                                                                    &                                                                                                         &                                                                                     & $f_k-f^*$ & $\mathcal{\cal O} \left({1}/{\sqrt{k}}\right)$            \\ \hline
	\multirow{2}{*}{\cite{Sabach17}}                            & \multirow{2}{*}{\begin{tabular}[c]{@{}c@{}}convex, continuously\\ differentiable, Lipschitz\end{tabular}} & \multirow{2}{*}{strongly convex, smooth}                                                                & \multirow{2}{*}{SAM}                                                                & $h_k-h^*$  & asymptotic                                                \\ \cline{5-6} 
	&                                                                                                                                    &                                                                                                         &                                                                                     & $f_k-f^*$ & $\mathcal{\cal O} \left({1}/{k}\right)$                   \\ \hline
	\multirow{2}{*}{\cite{Yousefian17}}                             & \multirow{2}{*}{convex, differentiable}                                                                                            & \multirow{2}{*}{strongly convex, differentiable}                                                        & \multirow{2}{*}{iter. regu.}                                                        & $h_k-h^*$  & $ \mathcal{\cal O} \left({1}/k^{1/6-\epsilon}\right) $  \\ \cline{5-6} 
	&                                                                                                                                    &                                                                                                         &                                                                                     & $f_k-f^*$ & asymptotic                                                \\ \hline
	\multirow{2}{*}{\cite{Rosasco18}}                                  & \multirow{2}{*}{\begin{tabular}[c]{@{}c@{}}convex, continuously\\ differentiable\end{tabular}}                                     & \multirow{2}{*}{\begin{tabular}[c]{@{}c@{}}strongly convex, continuously\\ differentiable\end{tabular}} & \multirow{2}{*}{iter. regu.}                                                        & $h_k-h^*$  &$ \mathcal{\cal O} \left({1}/k\right) $  \\ \cline{5-6} 
	&                                                                                                                                    &                                                                                                         &                                                                                     & $f_k-f^*$ & asymptotic                  \\ \hline
	\multirow{2}{*}{{\bf \scriptsize This work}} & \multirow{2}{*}{\begin{tabular}[c]{@{}c@{}} convex, \fy{nondifferentiable} \\ ({\bf finite sum form})\end{tabular}}                                                      & \multirow{2}{*}{strongly convex, \fy{nondifferentiable}}                                                                        & \multirow{2}{*}{\begin{tabular}[c]{@{}c@{}}incremental \\ iter. regu.\end{tabular}} & $h_k-h^*$  & asymptotic                                                \\ \cline{5-6} 
	&                                                                                                                                    &                                                                                                         &                                                                                     & $f_k-f^*$ & $ \mathcal{\cal O} \left({1}/k^{0.5-\epsilon}\right)  $ \\ \hline \hline
\end{tabular}}

	\label{fig:fiveplots}
\end{table*}

\subsection{Existing methods}
Problem \eqref{def:SL-INC}\fy{,} that is also referred to as \textit{hierarchical optimization}\fy{,} is a particular case of \textit{mathematical program with generalized equation (or equilibrium) constraint} \cite{Kocvara04,Pang96}. There has been a few approaches to tackle this problem. Note that in all approaches the following minimization problem and its minimizer have been extensively utilized.
\begin{definition}\label{def:reg minimizer}
\fy{Given a parameter $\lambda > 0$}, the regularized problem corresponding to \eqref{def:SL-INC}, \fy{is defined} as
	\begin{align}\label{reg form INC}
	\displaystyle \mbox{minimize}& \qquad f_\lambda (x) \triangleq f(x)+ \lambda h(x)\\
	\mbox{subject to} &\qquad x \in X\notag.
	\end{align} 
	Also, let $x_\lambda^*$ denote the unique minimizer of this problem.
	We may categorize the existing algorithms as follows.
	
\noindent {\bf(i) Exact regularization:} The regularization technique has been highly used in some applications such as signal processing with $h(x)= \|x\|_2^2$ or $h(x)= \|x\|_1$ \cite{Tikhonov77,Beck09}. This technique needs a proper parameter $\lambda$ which is difficult to determine in most of cases. To address this issue,  Mangasarian et al. \cite{Mangasarian79,Mangasarian85} introduced \textit{exact regularization}. A solution of  problem \eqref{reg form INC}  is called \textit{exact} when it is in the set $X^*$. The main drawback of this approach is that the threshold below which for any $\lambda$ the regularization \eqref{reg form INC} is exact, is very difficult to determine a priori \fy{(see \cite{Tseng07})}.

\noindent {\bf(ii) Iterative regularization:} In this approach\fy{,} the idea is to develop a single-loop scheme where the regularization parameter is updated iteratively during the algorithm. In \cite{Solodov07},  an \textit{explicit} descent algorithm is proposed, where problem \eqref{reg form INC} is solved as \fy{a single}-level unconstrained problem iteratively. \fy{In the smooth case,} the convergence is shown  when $\sum_{k=1}^{\infty} \lambda_k =\infty$ and $\lim_{k\rightarrow \infty} \lambda_k=0$.
For nonsmooth cases, a bundle method was proposed, which has a descent step for the weighted combination of objective functions in the lower and upper levels \cite{Solodov072}. Another algorithm called \textit{hybrid steepest descent method} (HSDM) and its extensions were developed in \cite{Yamada11,Neto11}. The main drawback of all these works is that no complexity analysis was provided for the \fy{underlying algrotihm}.  In \cite{Yousefian17}, the rate of $\mathcal{\cal O} (1/k^{1/6-\epsilon})$ is derived for a special case of problem \eqref{def:SL-INC}\fy{,} where the objective function is $\|\cdot\|_2$ and the constraint is solutions of a stochastic variational inequality problem. Recently, a primal-dual algorithm in \cite{Rosasco18} was offered that uses the idea of iterative regularization. Despite \fy{a sublinear rate in terms of $h$}, the algorithm can only be applied to continuously differentiable \fy{and small-scale (i.e., $m=1$)} regimes.

\noindent {\bf (iii) Minimal norm gradient:} In \cite{Beck14}, the \textit{minimal norm gradient} (MNG) method was developed for solving \fy{problem \eqref{def:SL-INC} with $m=1$}. The rate of $\mathcal{\cal O} (1/\sqrt{k})$ was derived for the convergence with respect to lower level problem. 
The main disadvantage of the MNG method is that it is a two-loop scheme where at each iteration a minimization problem should be solved. 

\noindent {\bf (iv) Sequential averaging:} The \textit{sequential averaging method} (SAM), developed in \cite{Xu04}, was \fy{employed} in \cite{Sabach17} for solving the problem in a more general setting. The proposed method is proved to have the rate of convergence of $\mathcal{\cal O} (1/k)$ in terms of the function $f$. The method is called \textit{Big-SAM}. \fy{Despite that it is a single-loop scheme, sequential averaging schemes require smoothness properties of the problem and seem not to lend themselves to distributed implementations.}
\subsection{Main Contributions}
\fy{For more details on the main distinctions between our work and the existing methods see Table \ref{fig:fiveplots}. In none of the existing methods,} the finite sum form for the lower level problem is \fy{considered}. The sum structure is very rampant in practice when we have separate objective functions related to different agents in a distributed setting. This is the case for example in machine learning for very large datasets \cite{Bottou05}, where each $f_i$ represents an agent that is cooperating with others. When the complete information of all the agents, i.e. summation of all (sub)gradients is not available, these agents can be treated distinctly. Due to its wide range of applications in distributed optimization, finite sum problem has been extensively studied. Among popular methods are \textit{incremental (sub)gradient} (IG) \cite{Bertsekas99,Nedich01} and \textit{incremental aggregated (sub)gradients} (IAG) \cite{Blatt07,Tseng14} for deterministic and \textit{stochastic average gradient} (SAG) \cite{Roux12}, SAGA \cite{Defazio14} and MISO methods \cite{Mairal13} for stochastic regimes. These algorithms have faster convergence and are computationally efficient in large-scale optimization since a very less amount of memory is required at each step in order to store only one agent's information and subsequently update the iterate based on that \cite{Ozdaglar17}. Despite the widespread use of these first-order methods, they do not address \fy{the} bilevel problem \eqref{def:SL-INC}. 


Motivated by the existing lack in the literature and inspired by the advantages of incremental approaches, in this paper, we let the lower level objective function to be a summation of $m$ components. Then we use the idea of incremental subgradient optimization to address problem \eqref{def:SL-INC}. We let functions in both levels to be nondifferentiable. We then prove the convergence of our proposed algorithm as well as the $\mathcal{\cal O} (1/k^{0.5-\epsilon})$ rate of convergence.

\begin{remark}
	An interesting research question that is remained as a future direction to our research is if we can establish the convergence of iterative regularized IAG method in solving problem \eqref{def:SL-INC} or similarly SAG, SAGA and MISO in stochastic regimes.
\end{remark}

{\bf Notation}
The inner product of two vectors $x,y \in \mathbb{R}^n$, is shown as $x^Ty$. Also, $\|\cdot\|$ denotes Euclidean norm known as $\|\cdot\|_2$. For a convex function $f$ with the domain dom$(f)$, any vector $g_f$ with $f(x)+g_f^T(y-x) \leq f(y)$ for all $x,y \in \hbox{dom}(f)$, is called a subgradient of $f$ at $x$. We let $\partial f(x)$ and $\partial  h(x)$ denote the set of all subgradients of functions $f$ and $h$ at $x$. Let $f^*$ be the optimal value and $X^*$ represent the set of all optimal solutions of \fy{the} lower level problem in \eqref{def:SL-INC} and $x^*$ shows any element of this set . Likewise, $x^*_h$ \fy{denotes} the optimal solution of problem \eqref{def:SL-INC}\fy{, and} $x^*_{\lambda}$ denotes the optimal solution of problem \eqref{reg form INC}. \fy{Also, we let $\mathcal{P}_X(x)$ denote the Euclidean projection of vector $x$ onto the set $X$.}

The rest of this paper is organized as follows. In Section \ref{sec: alg out}, we present the algorithm outline. Then, we discuss the convergence analysis in Section \ref{sec: conv ana}, and derive the convergence rate in Section \ref{sec:rate result}. We present the numerical results in Section \ref{sec:num}, and conclude in Section \ref{sec:rem}.
\section{Algorithm outline} \label{sec: alg out}
In this section, we introduce the \textit{iterative regularized incremental projected (sub)gradient} (IR-IG) for generating a sequence that converges to the unique optimal solution of \eqref{def:SL-INC}. See Algorithm \ref{algorithm:IRIG}.   
\begin{algorithm}
	\caption{IR-IG}
	\label{algorithm:IRIG}
	\begin{algorithmic}
		\STATE{\textbf{initialization:} Set an arbitrary initial point $x_0\in X$, $\bar x_0=x_0$, and $S_0=\gamma_0^r$ and pick $r<1$.}
		\FOR{$k=0,1,\ldots,N-1$}
		\STATE{Set $x_{k,0}=x_k$ and pick $\gamma_k>0$ and $\lambda_k>0$.}
		\FOR{$i=0,1,\ldots,m-1$}
		\STATE{Pick $g_{f_{i+1}}(x_{k,i})  \in \partial f_{i+1} (x_{k,i})$ and $g_h(x_{k,i}) \in \partial h(x_{k,i})$}.
		\STATE{Update $x_{k,i}$ using the following relation:}
		\begin{align}\label{mainstep}
		x_{k,i+1} := \mathcal{P}_X \left(x_{k,i} - \gamma_k\left(g_{f_{i+1}}(x_{k,i}) + \frac{\lambda_k}{m} g_h(x_{k,i})\right) \right).
		\end{align}
		\ENDFOR
		\STATE{Set $x_{k+1}:=x_{k,m}$.}
		\STATE{Update $S_k$ and $\bar x_{k}$ using:}
		\begin{align}
		S_{k+1}:=S_k+\gamma_{k+1}^r, \quad  
		\bar x_{k+1}:=\frac{S_k \bar x_k+\gamma_{k+1}^r x_{k+1}}{S_{k+1}}.\label{def:averagingII} 
		\end{align}
		\ENDFOR
		\RETURN $\bar x_{N}.$
	\end{algorithmic}
\end{algorithm}
IR-IG method includes two main steps. First, the agents update their iterates in an incremental fashion similar to the standard IG method. This step takes a circle around the all components of function $f$ to update the iterate. However, The main difference lies in the secondary objective function $h$, which is added by a vanishing multiplier $\lambda_k$. Second, we do averaging in order to accelerate the convergence speed of the algorithm. For this, we consider a weighted average sequence $\{\bar x_k\}$ defined as below: 
	\begin{align}\label{weighted alg INC}
	\bar x_{k+1} := \sum_{t=0}^{k} \psi_{t,k} x_t,  \hbox{ where } \psi_{t,k} \triangleq \frac{\gamma_t^r}{\sum_{i=0}^{k} \gamma_i^r},	
	\end{align}
	in which $r<1$ is a constant, controlling the weights. Note that \eqref{def:averagingII} in Algorithm \ref{algorithm:IRIG} follows from the relation \eqref{weighted alg INC} by applying induction,  (see e.g., Proposition 3 in \cite{Yousefian18}).
	
\end{definition}

\section{Convergence analysis}\label{sec: conv ana}
In this section, \fy{our goal is to show that} the generated sequence $\{\bar x_k\}$ by Algorithm \ref{algorithm:IRIG} converges to the unique optimal solution of problem \eqref{def:SL-INC} \fy{(see Theorem \ref{thm conv for xbar INC})}.
\begin{remark} \label{bound on gradient by beck}
\noindent (a) Note that from Theorem 3.16, pg. 42 of \cite{Beck17}, Assumption \ref{assum:properties} implies that there exist constants $C_f, C_h \in \mathbb{R}$ such that $\|g_{f_i}(x)\| \leq C_f$ and $\|g_{h}(x)\| \leq C_h$ for all $i=1,\cdots, m$ and $x\in X$, where $g_{f_i}(x) \in \partial f_i(x)$ and $g_{h}(x) \in \partial h(x)$.\\
\noindent (b) From Theorem 3.61, pg. 71 of \cite{Beck17}, functions $f_i$ and $h$ are Lipschitz over $X$ with parameters $C_f$ and $C_h$, respectively, i.e., for all $i=1,\cdots, m$ and $x,y\in X$
		\begin{align*}
		|f_i(x)-f_i(y)| \leq C_f \|x-y\|, \quad   |h(x)-h(y)| \leq C_h \|x-y\|.
		\end{align*}
\fy{\noindent	(c) Assumption \ref{assum:properties}(a,b) imply that the optimal solution set, $X^*$, is nonempty.}
\end{remark}

\vspace*{0.3 cm}
Here, we start with a lemma which helps bound the error of optimal solutions of the problem \eqref{reg form INC} for two different values of $\lambda$. We will make use of this lemma in the convergence analysis. The proof for this lemma can be done in a same fashion to that of Proposition 1 in \cite{Yousefian17}.
\begin{lemma} \label{lemma:err_inequalitygeneral}
	Let Assumption \ref{assum:properties} hold. Suppose $\{x_{\lambda_k}^* \}$ be the sequence of the optimal solutions of problem \eqref{reg form INC} with parameter $\lambda:= \lambda_k$. Then,
	\begin{itemize}
		
		\item[(a)] $\|x_{\lambda_k}^*-x_{\lambda_{k-1}}^*\|\leq \frac{C_h}{\mu_h}\left	|1-\frac{\lambda_{k-1}}{\lambda_k} \right|$.
		\item[(b)] If $\lambda_k \rightarrow 0$, then the sequence $\{x_{\lambda_k}^*\}$ converges to the unique optimal solution of problem \eqref{def:SL-INC}, i.e., $x_{h}^*$.
	\end{itemize}
\end{lemma}

To get started, we also need a recursive upper bound on the term $ \|x_{k+1} - x_{\lambda_k}^*  \|$. This is provided by the following lemma and will be used in Proposition \ref{conv x_k INC} to prove the convergence of sequence $\{x_k\}$ generated by the algorithm to $x_h^*$.

\begin{lemma} [{\bf A recursive error bound}] \label{lemma: an error bound INC}
	Let Assumption \ref{assum:properties} hold and $0<\mu_k\lambda_k \mu_h \leq 2m$. Then, for the sequence $\{x_k\}$ generated by Algorithm \ref{algorithm:IRIG} and for all $k>0$ we have
	\begin{align*}
	\left \|x_{k+1} - x_{\lambda_k}^* \right \|^2 & \leq \left( 1- \frac{\gamma_k \lambda_k \mu_h}{2m} \right) \left \|x_{k} - x_{\lambda_{k-1}}^* \right \|^2 \\ & + \frac{3m C_h^2}{\gamma_k \lambda_k \mu_h^3}\left	|1-\frac{\lambda_{k-1}}{\lambda_k} \right|^2 +  6 m^2 \gamma_k^2 (C_f^2 + \lambda_k^2 C_h^2),
	\end{align*}
	where $x_{\lambda_k}^*$ is the unique optimal solution of problem \eqref{reg form INC} with $\lambda:=\lambda_k$.
\end{lemma}
\begin{proof}
	Using \eqref{mainstep} and the nonexpansiveness property of projection, we have\\
	$
	\left \|x_{k,i+1} - x_{\lambda_k}^* \right \|^2  \\= \left \|\mathcal{P}_X \left(x_{k,i} - \gamma_k\left(g_{f_{i+1}}(x_{k,i}) +  \frac{\lambda_k}{m} g_h(x_{k,i})\right) \right) - \mathcal{P}_X(x_{\lambda_k}^*) \right \|^2 \\
	 \leq \left \|x_{k,i} - x_{\lambda_k}^* \right \|^2 + \gamma_k^2 \left \|g_{f_{i+1}}(x_{k,i}) + \frac{\lambda_k}{m} g_h(x_{k,i}) \right \|^2 \\-2\gamma_k \left ( g_{f_{i+1}}(x_{k,i}) +  \frac{\lambda_k}{m} g_h(x_{k,i})\right )^T\left( x_{k,i} - x_{\lambda_k}^*\right ).\\
	$
	\vspace{0.05cm}
	
	By boundedness of subgradients from Remark \ref{bound on gradient by beck}(a), the definition of subgradient for $f_{i+1}$, and the strong convexity of $h$, we obtain
	\begin{align*}
	& \left \|x_{k,i+1} - x_{\lambda_k}^* \right \|^2\\ & \leq \left \|x_{k,i} - x_{\lambda_k}^* \right \|^2 +2 \gamma_k^2 (C_f^2 + \lambda_k^2 C_h^2)  - 2\gamma_k \left( f_{i+1}(x_{k,i})- f_{i+1}(x_{\lambda_k}^*) \right) \\
	& -  \frac{2\gamma_k \lambda_k}{m} \left( h(x_{k,i}) - h(x_{\lambda_k}^*) \right) - \frac{\gamma_k \lambda_k \mu_h}{m} \left \|x_{k,i} - x_{\lambda_k}^* \right \|^2 \\
	& = \left( 1- \frac{\gamma_k \lambda_k \mu_h}{m} \right) \left \| x_{k,i} - x_{\lambda_k}^* \right \|^2 - 2\gamma_k \left( f_{i+1}(x_{k,i}) + \frac{\lambda_k}{m} h(x_{k,i}) \right) \\& + 2\gamma_k \left(f_{i+1}(x_{\lambda_k}^*)+ \frac{\lambda_k}{m} h(x_{\lambda_k}^*) \right) +2 \gamma_k^2 (C_f^2 + \lambda_k^2 C_h^2).
	\end{align*}
	Taking summation from both sides over $i$, using $x_{k,0}=x_k,  x_{k,m}=x_{k+1}$, and that $\gamma_k \lambda_k \mu_h >0$, we obtain
	\fy{\begin{align}
	&\sum_{i=0}^{m-1} \left \|x_{k,i+1} - x_{\lambda_k}^* \right \|^2 \nonumber \\
	&\leq \left( 1- \frac{\gamma_k \lambda_k \mu_h}{m} \right) \left \|x_k - x_{\lambda_k}^* \right \|^2 + \sum_{i=1}^{m-1} \left \| x_{k,i} - x_{\lambda_k}^* \right \|^2 \nonumber \\& + 2 m \gamma_k^2 (C_f^2 + \lambda_k^2 C_h^2) 
	- 2\gamma_k \sum_{i=0}^{m-1} \left( f_{i+1}(x_{k,i}) + \frac{\lambda_k}{m} h(x_{k,i}) \right) \nonumber\\&+ 2\gamma_k \left(f(x_{\lambda_k}^*)+ \lambda_k h(x_{\lambda_k}^*) \right), \label{ineq1 INC}
	\end{align}}
	where we used the definition of function $f$ in the second inequality. Now by rearranging the terms and adding and subtracting $f(x_k)+\lambda_k h(x_k)$ we obtain
	\fy{\begin{align*}
	&\left \|x_{k+1} - x_{\lambda_k}^* \right \|^2 \\ &\leq \left( 1- \frac{\gamma_k \lambda_k \mu_h}{m} \right) \left \|x_k - x_{\lambda_k}^* \right \|^2 + 2 m \gamma_k^2 (C_f^2 + \lambda_k^2 C_h^2)\\
	& - 2\gamma_k \sum_{i=0}^{m-1} \left((f_{i+1}(x_{k,i}) - f_{i+1}(x_{k}) ) + \frac{\lambda_k}{m} (h(x_{k,i})  -h(x_{k}))  \right) \\& + 2\gamma_k \underbrace{\left(f(x_{\lambda_k}^*)+ \lambda_k h(x_{\lambda_k}^*) -f(x_k)-\lambda_k h(x_k) \right)}_{Term1}\\
	&\leq \left( 1- \frac{\gamma_k \lambda_k \mu_h}{m} \right) \left \|x_k - x_{\lambda_k}^* \right \|^2 + 2 m \gamma_k^2 (C_f^2 + \lambda_k^2 C_h^2)\\
	& + 2\gamma_k \sum_{i=0}^{m-1} \left( \underbrace{|f_{i+1}(x_{k,i}) - f_{i+1}(x_{k}) |}_{Term2} + \frac{\lambda_k}{m} \underbrace{|h(x_{k,i})  -h(x_{k})|}_{Term3}  \right),
	\end{align*}}
	where $Term1\leq 0$ is used due to optimality of $x_{\lambda_k}^*$ for $f+\lambda_k h$. Also, from Remark \ref{bound on gradient by beck}(b) we know that $Term2 \leq C_f\|x_{k,i} - x_k\|$ and $Term3 \leq C_h\|x_{k,i} - x_k\|$. So, We have
	\begin{align}\label{incrementalineq1}
	&\left \|x_{k+1} - x_{\lambda_k}^* \right \|^2  \leq \left( 1- \frac{\gamma_k \lambda_k \mu_h}{m} \right) \left \|x_{k} - x_{\lambda_k}^* \right \|^2  \nonumber \\ &+ 2 m \gamma_k^2 (C_f^2 + \lambda_k^2 C_h^2)
	+ 2 (C_f+\lambda_k C_h) \gamma_k \sum_{i=0}^{m-1} \| x_{k,i} - x_k\|.
	\end{align}
	Next, we find an upper bound for $\| x_{k,i} - x_k\|$. We have\\
	$
	\| x_{k,1} - x_k\|\\ =\left \|\mathcal{P}_X \left(x_{k,0} - \gamma_k\left(g_{f_1}(x_{k,0}) + \frac{\lambda_k}{m} g_h(x_{k,0})\right) \right)- \mathcal{P}_X(x_k) \right \|\\
	\leq \gamma_k \left \|g_{f_1}(x_{k,0}) + \frac{\lambda_k}{m} g_h(x_{k,0})\right \| \leq \gamma_k \left( C_f + \frac{\lambda_k}{m} C_h \right).\\
	$
	
		\vspace{0.05cm}
	For $i>0$, in a similar way, we have
	\begin{align*}
	\| x_{k,i+1} - x_k\| 
	\leq \| x_{k,i} - x_k\|+ \gamma_k \left( C_f + \frac{\lambda_k}{m} C_h \right).
	\end{align*}
	So for $i=0,1, \cdots, m-1$, we have
	\begin{align} \label{ineq2 INC}
	\| x_{k,i+1} - x_k\| &\leq (i+1)\gamma_k \left( C_f + \frac{\lambda_k}{m} C_h \right) \nonumber \\&\leq  (i+1)\gamma_k \left( C_f + \lambda_k C_h \right).
	\end{align}
	Combining this with \eqref{incrementalineq1}, we will obtain
	\begin{align}\label{incrementalineq2}
	\left \|x_{k+1} - x_{\lambda_k}^* \right \|^2  
	&\leq \left( 1- \frac{\gamma_k \lambda_k \mu_h}{m} \right) \left \|x_{k} - x_{\lambda_k}^* \right \|^2 \nonumber \\ &+ 6 m^2 \gamma_k^2 (C_f^2 + \lambda_k^2 C_h^2).
	\end{align}
	Next, we relate $x_k$ to $x_{\lambda_{k-1}}^*$. We have
	\begin{align*}
	&\left \|x_{k} - x_{\lambda_k}^* \right \|^2=\left \|x_{k} - x_{\lambda_{k-1}}^* \right\|^2 + \left \|x_{\lambda_k}^* - x_{\lambda_{k-1}}^* \right\|^2 \\ & + \underbrace{2 \left(x_{k} - x_{\lambda_{k-1}}^* \right)^T \left(x_{\lambda_{k-1}}^* - x_{\lambda_k}^*  \right)}_{Term4}.
	\end{align*}
	Applying the fact that $2a^Tb \leq \|a\|^2/\alpha + \alpha \|b\|^2 $ for all $a,b \in \mathbb{R}^n$ and $\alpha>0$ for $Term4$ when $\alpha = 2m/{\gamma_k \lambda_k \mu_h}$, we obtain
	\begin{align*}
	&\left \|x_{k} - x_{\lambda_k}^* \right \|^2 =\left \|x_{k} - x_{\lambda_{k-1}}^* \right\|^2 + \left \|x_{\lambda_k}^* - x_{\lambda_{k-1}}^* \right\|^2 \\ & + \frac{\gamma_k \lambda_k \mu_h}{2m} \left \|x_{k} - x_{\lambda_{k-1}}^* \right\|^2 + \frac {2m} {\gamma_k \lambda_k \mu_h} \left \|x_{\lambda_k}^* - x_{\lambda_{k-1}}^* \right\|^2\\
	& = \left( 1+ \frac{\gamma_k \lambda_k \mu_h}{2m} \right) \left \|x_{k} - x_{\lambda_{k-1}}^* \right\|^2 + \left(1+ \frac{2m}{\gamma_k \lambda_k \mu_h}\right)\left \|x_{\lambda_k}^* - x_{\lambda_{k-1}}^* \right\|^2. 
	\end{align*}
	Using Lemma \ref{lemma:err_inequalitygeneral}(a), we obtain
	\begin{align*}
	&\left \|x_{k} - x_{\lambda_k}^* \right \|^2  \leq  \left( 1+ \frac{\gamma_k \lambda_k \mu_h}{2m} \right) \left \|x_{k} - x_{\lambda_{k-1}}^* \right\|^2 \\ & + \left(1+ \frac{2m}{\gamma_k \lambda_k \mu_h}\right) \frac{C_h^2}{\mu_h^2}\left	|1-\frac{\lambda_{k-1}}{\lambda_k} \right|^2.
	\end{align*}
	Plugging this inequality into \eqref{incrementalineq2} we obtain
	\begin{align*}
	&\left \|x_{k+1} - x_{\lambda_k}^* \right \|^2  \leq \left( 1- \frac{\gamma_k \lambda_k \mu_h}{m} \right) \left( 1+ \frac{\gamma_k \lambda_k \mu_h}{2m} \right) \left \|x_{k} - x_{\lambda_{k-1}}^* \right \|^2 \\& +\left( 1- \frac{\gamma_k \lambda_k \mu_h}{m} \right) \left(1+ \frac{2m}{\gamma_k \lambda_k \mu_h}\right) \frac{C_h^2}{\mu_h^2}\left	|1-\frac{\lambda_{k-1}}{\lambda_k} \right|^2 \\& +  6 m^2 \gamma_k^2 (C_f^2 + \lambda_k^2 C_h^2). 
	\end{align*}
	The desired result is obtained \fy{from} $0<\mu_k\lambda_k \mu_h \leq 2m$.
\end{proof}
We will make use of the following result in Proposition \ref{conv x_k INC}.
\begin{lemma}[{Lemma 10, pg. 49, \cite{Polyak87}}]  \label {lemma conv lemma}
	Let $\{\nu_k\}$ be a sequence of nonnegative scalars and let $\{\alpha_k\}$ and $\{\beta_k\}$ be scalar sequences such that: 
	\begin{align*}
	&\nu_{k+1} \leq (1-\alpha_k)\nu_k + \beta_k \qquad \hbox{for all } k \geq 0,\\
	&0 \leq \alpha_k \leq 1, \ \beta_k \geq0, \ \sum_{k=0}^{\infty}\alpha_k=\infty, \ \sum_{k=0}^{\infty}\beta_k <\infty,  \ \lim_{k \to \infty} \frac{\beta_k}{\alpha_k}=0.
	\end{align*}
	Then, $\lim_{k \to \infty} \nu_k = 0$.
\end{lemma}
\begin{assumption}\label{assumptions for conv INC} Assume that for all $k \geq 0$ we have
	
	\noindent$(a)\{\gamma_k\}$ and $\{\lambda_k\}$ are non-increasing positive sequences with $\gamma_0\lambda_0 \leq \frac{2m}{\mu_h}.$    \\  $(b) \sum_{k=0}^{\infty}\gamma_k\lambda_k= \infty.$   
	$ \qquad (c) \sum_{k=0}^{\infty}\frac{1}{\gamma_k\lambda_k}\left(\frac{\lambda_{k-1}}{\lambda_k}-1\right)^2 <\infty.$\\ $(d) \sum_{k=0}^{\infty} \gamma_k^2 < \infty.$         \\
	$(e) \lim_{k\to \infty} \frac{1}{\gamma_k^2\lambda_k^2}\left(\frac{\lambda_{k-1}}{\lambda_k}-1\right)^2=0.$ $\qquad (f) \lim_{k\to \infty} \frac{\gamma_k}{\lambda_k}=0.$
	
\end{assumption}
\begin{proposition}[{\bf Convergence of $\mathbf{\{x_k\}}$}]\label{conv x_k INC} 
	Consider problem \eqref{def:SL-INC}. Let Assumption \ref{assum:properties} and \ref{assumptions for conv INC} hold and $\{x_k\}$ be generated by Algorithm \ref{algorithm:IRIG}. Then,
	\begin{itemize}
		\item[(a)] $\lim_{k \rightarrow \infty} \|x_k-x_{\lambda_{k-1}}^*\|^2=0$.
		\item[(b)] If $\lim_{k \rightarrow \infty} \lambda_k=0$, $x_k$ converges to the unique optimal solution of problem \eqref{def:SL-INC}, i.e., $x_h^*$.
	\end{itemize}
\end{proposition}
\begin{proof}
	\noindent (a) Consider the result from Lemma \ref{lemma: an error bound INC}. We let
	\begin{align*}
	\nu_k &\triangleq \|x_k - x_{\lambda_{k-1}}^*\|^2, \ \alpha_k \triangleq \frac{\gamma_k\lambda_k \mu_h}{2m}\\ \beta_k &\triangleq \frac{3m C_h^2}{\gamma_k \lambda_k \mu_h^3}\left(1-\frac{\lambda_{k-1}}{\lambda_k} \right)^2 +  6 m^2 \gamma_k^2 (C_f^2 + \lambda_k^2 C_h^2).
	\end{align*}
	From Assumption \ref{assumptions for conv INC}(a,b,c), since $\{\gamma_k\}$ and $\{\lambda_k\} $ are positive and $\gamma_0\lambda_0 \leq 2m/\mu_h$, we have $0 \leq \alpha_k \leq 1$ and  $ \beta_k \geq0$ and also $\sum_{k=1}^{\infty}\alpha_k=\infty$ and $\sum_{k=1}^{\infty}\beta_k <\infty$. To show that all the necessary assumptions for Lemma \ref{lemma conv lemma} are satisfied, we have
	\begin{align*}
	\lim_{k \to \infty} \frac{\beta_k}{\alpha_k}
	&=\frac{6m^2 C_h^2}{\mu_h^4}\lim_{k \to \infty} \frac{1}{\gamma_k^2\lambda_k^2}\left(\frac{\lambda_{k-1}}{\lambda_k}-1\right)^2 +\frac{12m^3C_f^2}{\mu_h}\lim_{k\to \infty} \frac{\gamma_k}{\lambda_k} \\& +
	\frac{12m^3C_h^2}{\mu_h}\lim_{k\to \infty} \gamma_k\lambda_k.
	\end{align*}
	Considering Assumption \ref{assumptions for conv INC}(e,f), we only need to show that $\lim_{k\to \infty} \gamma_k\lambda_k=0$. Since $\{\lambda_k\}$ is non-increasing for all $k\geq 0$, we have $\lambda_0^2 {\gamma_k}/{\lambda_k} \geq \gamma_k\lambda_k$. So by Assumption \ref{assumptions for conv INC}(f), $\lim_{k \to \infty} \gamma_k\lambda_k=0$. Consequently $\lim_{k \to \infty} \frac{\beta_k}{\alpha_k}=0$. Now Lemma \ref{lemma conv lemma} can be applied. We have
	\begin{align*}
	\lim_{k \to \infty} \nu_k = \lim_{k \to \infty} \|x_k - x_{\lambda_{k-1}}^*\|^2=0.
	\end{align*}
	
	\noindent (b) Applying the triangular inequality, we obtain
	\begin{align*}
	\|x_k -x_h^*\|^2 \leq 2\|x_k - x_{\lambda_{k-1}}^* \|^2 + 2 \|x_{\lambda_{k-1}}^*-x_h^*\|^2, \ \hbox{for all } k \geq 0.
	\end{align*}
	From part (a), $\|x_k -x_h^*\|^2$ converges to zero. Also, from Lemma \ref{lemma:err_inequalitygeneral}(b), we know that when $\lambda_k \rightarrow 0$ the sequence $\{x_{\lambda_k}^*\}$ converges to the unique optimal solution of problem \eqref{def:SL-INC}, i.e., $x_h^*$. Therefore the result holds.
\end{proof}
To have previous proposition work, we require that sequences $\{\gamma_k\}$ and $\{\lambda_k\}$ satisfy Assumption \ref{assumptions for conv INC}. Below, we provide a set of feasible sequences for this assumption. The proof is analogous to that of Lemma 5 in \cite{Yousefian17}.

\begin{lemma}\label{lemma condition for sequences INC}
	Assume $\{\gamma_k\}$ and $\{\lambda_k\}$ are sequences such that $\gamma_k=\frac{\gamma_0}{(k+1)^a}$ and $\lambda_k=\frac{\lambda_0}{(k+1)^b}$ where $a,b,\gamma_0$ and $\lambda_0$ are positive scalars and $\gamma_0\lambda_0 \leq \frac{2m}{\mu_h}$. If $a>b$, $a>0.5$ and $a+b<1$,  then the sequences $\{\gamma_k\}$ and $\{\lambda_k\}$ satisfy Assumption \ref{assumptions for conv INC}. 
\end{lemma}
The following is a useful lemma in proving convergence that we will apply in Theorem \ref{thm conv for xbar INC}.

\begin{lemma}[{Theorem 6, pg.\ 75 of~\cite{Knopp51}}] \label{lemma thm for avr}
	Let $\{u_t\}\subset \mathbb{R}^n$ be a convergent sequence with the limit point $\hat u\in\mathbb{R}^n$ and let $\{\alpha_k\}$ be a
	sequence of positive numbers where $\sum_{k=0}^\infty \alpha_k=\infty$. Suppose $\{v_k\}$ is given by 
	$v_k\triangleq \left({\sum_{t=0}^{k-1} \alpha_t u_t}\right)/{\sum_{t=0}^{k-1} \alpha_t}$ for all $k\ge1$. Then, $\lim_{k \rightarrow \infty} v_k=\hat u.$
\end{lemma}
Now, we can illustrate our ultimate goal in this section which is showing the convergence of the sequence $\{\bar{x_k}\}$ generated by Algorithm \ref{algorithm:IRIG} to $x_h^*$
\begin{theorem}[{\bf Convergence of $\mathbf{\{\bar x_k\}}$}] \label{thm conv for xbar INC}
	Consider problem \eqref{def:SL-INC}. Let Assumption \ref{assum:properties} hold. Also assume $\{\gamma_k\}$ and $\{\lambda_k\}$ are sequences such that $\gamma_k=\frac{\gamma_0}{(k+1)^a}$ and $\lambda_k=\frac{\lambda_0}{(k+1)^b}$ where $a,b,\gamma_0$ and $\lambda_0$ are positive scalars and $\gamma_0\lambda_0 \leq \frac{2m}{\mu_h}$. Let $\{\bar x_k\}$ be generated by Algorithm \ref{algorithm:IRIG}.  If $a>b$, $a>0.5$, $a+b<1$ and $ar \leq 1$, then $\{\bar x_k\}$ converges to $x_h^*$.
\end{theorem}
\begin{proof}
	Considering the given assumptions, by Lemma \ref{lemma condition for sequences INC} we can see that Assumption \ref{assumptions for conv INC} holds. We have
	\begin{align*}
	\|\bar x_{k+1}-x_h^*\|=\left\|\sum_{t=0}^{k} \psi_{t,k} x_t-\sum_{t=0}^{k}\psi_{t,k} x_h^*\right\|\leq \sum_{t=0}^{k} \psi_{t,k}\|x_t-x_h^*\|,
	\end{align*}
	applying $\sum_{t=0}^{k} \psi_{t,k}=1$ from \eqref{weighted alg INC} and the triangular inequality.
	Now consider the definition of $\psi_k$ and let $\alpha_t \triangleq \gamma_t^r$, $u_t \triangleq \|x_t-x_h^*\|$ and $v_{k+1} \triangleq \sum_{t=0}^{k} \psi_{t,k} \|x_t-x_h^*\|$. Since $ar \leq1$ we have $\sum_{t=0}^\infty \alpha_t=\sum_{t=0}^\infty \gamma_t^{r}=\sum_{t=0}^\infty (t+1)^{-ar}=\infty$. The sequence $\{\lambda_t \}$ is decreasing to zero due to $b>0$, So, from Proposition \ref{conv x_k INC}(b), $u_t=\|x_t-x_h^*\|$ converges to zero. Therefore, for $\hat u=0$ we can apply Lemma \ref{lemma thm for avr} and thus, $\|\bar x_{k+1}-x_h^*\|$ converges to zero.
\end{proof}


\section{Rate analysis}\label{sec:rate result}
In this section, we first find an error bound with respect to the optimal values of the lower level function $f$ which indeed shows the feasibility of \fy{the problem} \eqref{def:SL-INC}. Then, we apply it to derive a convergence rate for the algorithm.
\begin{lemma} \label{lemma averaging} Consider the sequence $\{\bar x_N\}$ generated by Algorithm \ref{algorithm:IRIG}. Let Assumption \ref{assum:properties} hold and $\{\gamma_k\}$ and $\{\lambda_k\}$ be positive and non-increasing sequences. Then, for all $N\geq 1$ and $z\in X$ we have
	\begin{align*}
	&f(\bar x_N)-f^*\leq \left(\sum_{k=0}^{N-1} \gamma_k^r\right)^{-1} \left( m \sum_{k=0}^{N-1} \gamma_k^{r+1} (C_f^2 + \lambda_k^2 C_h^2) \right. \cr\\& \left. 
	+m^2 C_f \sum_{k=0}^{N-1} \gamma_k^{r+1} \left( C_f + \lambda_k C_h \right) + 2 M_h  \sum_{k=0}^{N-1} \gamma_k^r \lambda_k + 2M^2\gamma_{N-1}^{r-1} \right),
	\end{align*}
	where $M_h, M$ are scalars such that $\|h(x)\|\leq M_h$, $\|x\| \leq M$ for all $x \in X$.
\end{lemma}
\begin{proof}
	\fy{Similar to} relation \eqref{ineq1 INC}, we \fy{can have}
	\begin{align*}
	&\left \|x_{k+1} - x^* \right \|^2  \leq \left \|x_k - x^* \right \|^2 + 2 m \gamma_k^2 (C_f^2 + \lambda_k^2 C_h^2) \\
	& - 2\gamma_k \sum_{i=0}^{m-1} \left( f_{i+1}(x_{k,i}) + \frac{\lambda_k}{m} h(x_{k,i}) \right) + 2\gamma_k \left(f^*+ \lambda_k h(x^*) \right) \\
	& \leq \left \|x_k - x^* \right \|^2 + 2 m \gamma_k^2 (C_f^2 + \lambda_k^2 C_h^2) \\
	& - 2\gamma_k \sum_{i=0}^{m-1} f_{i+1}(x_{k,i}) + 2\gamma_k f^*+ 4\gamma_k \lambda_k M_h.
	\end{align*}
	Adding and subtracting $f(x_k)$, we obtain
	\begin{align*}
	&\left \|x_{k+1} - x^* \right \|^2  \leq \left \|x_k - x^* \right \|^2 + 2 m \gamma_k^2 (C_f^2 + \lambda_k^2 C_h^2) \\
	& - 2\gamma_k \sum_{i=0}^{m-1} \left( f_{i+1}(x_{k,i}) - f_i(x_k) \right) + 2\gamma_k \left( f^* - f(x_k)\right)+ 4\gamma_k \lambda_k M_h\\
	& \leq \left \|x_k - x^* \right \|^2 + 2 m \gamma_k^2 (C_f^2 + \lambda_k^2 C_h^2) \\
	& + 2\gamma_k \sum_{i=0}^{m-1} \left| f_{i+1}(x_{k,i}) - f_i(x_k) \right| + 2\gamma_k \left( f^* - f(x_k)\right)+ 4\gamma_k \lambda_k M_h.
	\end{align*}
	Applying Remark \ref{bound on gradient by beck}(b), $| f_{i+1}(x_{k,i}) - f_i(x_k)| \leq C_f\|x_{k,i} - x_k\|$, we have
	\begin{align*}
	&\left \|x_{k+1} - x^* \right \|^2  \leq \left \|x_k - x^* \right \|^2 + 2 m \gamma_k^2 (C_f^2 + \lambda_k^2 C_h^2) \\
	& + 2 C_f \gamma_k \sum_{i=0}^{m-1} \|x_{k,i} - x_k\| + 2\gamma_k \left( f^* - f(x_k)\right)+ 4\gamma_k \lambda_k M_h.
	\end{align*}
	Using the inequality \eqref{ineq2 INC}, we obtain
	\begin{align}
	&\left \|x_{k+1} - x^* \right \|^2  -  \left \|x_k - x^* \right \|^2  \leq 2\gamma_k \left( f^* - f(x_k)\right) \nonumber\\
	& + 2 m \gamma_k^2 (C_f^2 + \lambda_k^2 C_h^2)   + 2 m^2 C_f \gamma_k^2 \left( C_f + \lambda_k C_h \right) + 4\gamma_k \lambda_k M_h.\label{ineq3 INC} 
	\end{align}
	Multiplying both sides by $\gamma_k^{r-1}$ and adding and subtracting  $\gamma_{k-1}^{r-1} \|x_k - x^*\|^2$ to the left hand side, we have
	\begin{align*}
	&\gamma_k^{r-1} \left \|x_{k+1} - x^* \right \|^2  - \gamma_{k-1}^{r-1} \|x_k - x^*\|^2  + \left( \gamma_{k-1}^{r-1}-\gamma_k^{r-1}\right)  \left \|x_k - x^* \right \|^2 \\& \leq 2 \gamma_k^r \left( f^* - f(x_k)\right) + 2 m \gamma_k^{r+1} (C_f^2 + \lambda_k^2 C_h^2) \\
	& + 2 m^2 C_f \gamma_k^{r+1} \left( C_f + \lambda_k C_h \right) + 4 \gamma_k^r \lambda_k M_h.
	\end{align*}
	Since $\{\gamma_k\}$ in non-increasing and $r<1$ we have $ \gamma_{k-1}^{r-1} \leq \gamma_k^{r-1} $. Also, by the triangle inequality $\left \|x_k - x^* \right \|^2 \leq 2\|x_k\|^2+2\|x^*\|^2 \leq 4M^2$. So, we obtain
	\begin{align*}
	&\gamma_k^{r-1} \left \|x_{k+1} - x^* \right \|^2  - \gamma_{k-1}^{r-1} \|x_k - x^*\|^2  + 4M^2\left( \gamma_{k-1}^{r-1}-\gamma_k^{r-1}\right)\\ & \leq 2 \gamma_k^r \left( f^* - f(x_k)\right) + 2 m \gamma_k^{r+1} (C_f^2 + \lambda_k^2 C_h^2) \\
	& + 2 m^2 C_f \gamma_k^{r+1} \left( C_f + \lambda_k C_h \right) + 4 \gamma_k^r \lambda_k M_h.
	\end{align*}
	Taking summation over $k=1,2, \cdots, N-1$, we obtain
	\begin{align*}
	&\gamma_{N-1}^{r-1} \left \|x_N - x^* \right \|^2  - \gamma_{0}^{r-1} \|x_1 - x^*\|^2 + 4M^2\left( \gamma_{0}^{r-1}-\gamma_{N-1}^{r-1}\right)   \\& \leq 2 \sum_{k=1}^{N-1} \gamma_k^r \left( f^* - f(x_k)\right)  \label{ineq4 INC}+ 2 m \sum_{k=1}^{N-1} \gamma_k^{r+1} (C_f^2 + \lambda_k^2 C_h^2) 
	\\ & + 2 m^2 C_f \sum_{k=1}^{N-1} \gamma_k^{r+1} \left( C_f + \lambda_k C_h \right) + 4 M_h \sum_{k=1}^{N-1} \gamma_k^r \lambda_k .
	\end{align*}
	Removing non-negative terms from the left-hand side of the preceding inequality, we have
	\begin{align}
	&- \gamma_{0}^{r-1} \|x_1 - x^*\|^2  - 4M^2\gamma_{N-1}^{r-1}  \leq 2 \sum_{k=1}^{N-1} \gamma_k^r \left( f^* - f(x_k)\right) \nonumber \\ & + 2 m \sum_{k=1}^{N-1} \gamma_k^{r+1} (C_f^2 + \lambda_k^2 C_h^2) 
	+ 2 m^2 C_f \sum_{k=1}^{N-1} \gamma_k^{r+1} \left( C_f + \lambda_k C_h \right) \nonumber \\& + 4 M_h \sum_{k=1}^{N-1} \gamma_k^r \lambda_k . 
	\end{align}
	From \eqref{ineq3 INC} for $k=0$, we obtain
	\begin{align*}
	\left \|x_1 - x^* \right \|^2  & \leq 2\gamma_0 \left( f^* - f(x_0)\right) + 2 m \gamma_0^2 (C_f^2 + \lambda_0^2 C_h^2)  \\
	& + 2 m^2 C_f \gamma_0^2 \left( C_f + \lambda_0 C_h \right) + 4\gamma_0 \lambda_0 M_h + 4M^2. 
	\end{align*}
	By multiplying both sides of the preceding inequality by $\gamma_{0}^{r-1}$ and summing it with the relation \eqref{ineq4 INC}, we obtain
	\begin{align*}
	&- 4M^2\gamma_{N-1}^{r-1}  \leq 2 \sum_{k=0}^{N-1} \gamma_k^r \left( f^* - f(x_k)\right)  + 4 M_h \sum_{k=0}^{N-1} \gamma_k^r \lambda_k \\& + 2 m \sum_{k=0}^{N-1} \gamma_k^{r+1} (C_f^2 + \lambda_k^2 C_h^2) 
	 + 2 m^2 C_f \sum_{k=0}^{N-1} \gamma_k^{r+1} \left( C_f + \lambda_k C_h \right)  . 
	\end{align*}
	Rearranging the terms we have
	\begin{align*}
	&2 \sum_{k=0}^{N-1} \gamma_k^r \left( f(x_k) - f^*\right)  \leq  2 m \sum_{k=0}^{N-1} \gamma_k^{r+1} (C_f^2 + \lambda_k^2 C_h^2) \\& 
	+ 2 m^2 C_f \sum_{k=0}^{N-1} \gamma_k^{r+1} \left( C_f + \lambda_k C_h \right) + 4 M_h \sum_{k=0}^{N-1} \gamma_k^r \lambda_k + 4M^2\gamma_{N-1}^{r-1} . 
	\end{align*}
	Now we divide both sides by $2\sum_{k=0}^{N-1} \gamma_k^r$ and use the definition of $\psi_{k,N-1}$ in \eqref{weighted alg INC},
	\begin{align*}
	&\sum_{k=0}^{N-1} \psi_{k,N-1} \left( f(x_k) - f^*\right) \\& \leq   \left(\sum_{k=0}^{N-1} \gamma_k^r\right)^{-1} \left( m \sum_{k=0}^{N-1} \gamma_k^{r+1} (C_f^2 + \lambda_k^2 C_h^2) \right. \cr\\& \left. 
	+m^2 C_f \sum_{k=0}^{N-1} \gamma_k^{r+1} \left( C_f + \lambda_k C_h \right) + 2 M_h  \sum_{k=0}^{N-1} \gamma_k^r \lambda_k + 2M^2\gamma_{N-1}^{r-1} \right). 
	\end{align*}
	We know $\sum_{k=0}^{N-1} \psi_{k,N-1}=1$  also $f(\bar x_N) \leq \sum_{k=0}^{N-1} \psi_{k,N-1}f(x_k)$ because of convexity of $f$. So, we obtain the desired result.
\end{proof}
The following lemma, will be used to find a convergence rate statement in Theorem \ref{thm rate}.

\begin{lemma} [Lemma 9, page 418 in \cite{Yousefian17}]\label{lemma:ineqHarmonic}
	For any scalar $\alpha\neq -1$ and integers $\ell$ and $N$ where $0\leq \ell \leq N-1$, we have
	\begin{align*}
	\frac{N^{\alpha+1}-(\ell+1)^{\alpha+1}}{\alpha+1}&\leq \sum_{k=\ell}^{N-1}(k+1)^\alpha \\& \leq (\ell+1)^\alpha+\frac{(N+1)^{\alpha+1}-(\ell+1)^{\alpha+1}}{\alpha+1}.
	\end{align*}
\end{lemma}
In the following theorem\fy{,} we present a rate statement for Algorithm \ref{algorithm:IRIG}. 
\begin{theorem} [\bf{A rate statement for Algorithm \ref{algorithm:IRIG}}] \label{thm rate}
	Assume  $\{\bar x_N\}$ is generated by Algorithm \ref{algorithm:IRIG} to solve problem \eqref{def:SL-INC}. Let Assumption \ref{assum:properties} and \ref{assumptions for conv INC} hold and also $0<\epsilon<0.5$ and $r<1$ be arbitrary constants. Assume for $0<\epsilon<0.5$, $\{\gamma_k\}$ and $\{\lambda_k\}$ are sequences defined as
	$$
			\gamma_k=\frac{\gamma_0}{(k+1)^{0.5+0.5\epsilon}} \hbox{ and } \lambda_k=\frac{\lambda_0}{(k+1)^{0.5-\epsilon}},
	$$
	such that $\gamma_0$ and $\lambda_0$ are positive scalars and $\gamma_0\lambda_0 \mu_h \leq 2m$. Then,\\
		\noindent (a) The sequence $\{\bar x_N\}$ converges to the unique optimal solution of problem \eqref{def:SL-INC}, i.e.,  $x_h^*$.\\
		\noindent (b) $f(\bar x_N)$ converges to $f^*$ with the rate ${\cal O}\left(1/N^{0.5-\epsilon}\right)$.
\end{theorem}

\begin{proof} Throughout, we set $a:=0.5+0.5\epsilon$, $b:=0.5-\epsilon$.
	
	\noindent (a) From the values of $a$ and $b$, and that $r<1$ and $0<\epsilon<0.5$, we have
	\begin{align*}
	a>b>0,\ a>0.5, \ a+b=1-0.5\epsilon<1,\\ ar= 0.5(1+\epsilon)r <0.5(1.5)=0.75<1.
	\end{align*}
	This implies that all conditions of Theorem \ref{thm conv for xbar INC} are satisfied. Therefore, $\{\bar x_N\}$ converges to $x_h^*$ almost surely.
	
	\noindent (b)
Since $\{\lambda_k\}$ is a non-increasing sequence from Lemma \ref{lemma averaging} we have
\begin{align*}
&f(\bar x_N) - f^*  \leq   \left(\sum_{k=0}^{N-1} \gamma_k^r\right)^{-1} \left( m \left(C_f^2 + \lambda_0^2 C_h^2\right) \sum_{k=0}^{N-1} \gamma_k^{r+1}  \right. \cr\\& \left. 
+m^2 C_f\left( C_f + \lambda_0 C_h \right) \sum_{k=0}^{N-1} \gamma_k^{r+1} + 2 M_h   \sum_{k=0}^{N-1} \lambda_k \gamma_k^r + 2M^2\gamma_{N-1}^{r-1} \right). 
\end{align*}
We have $\gamma_k=\gamma_0/(k+1)^a$ and $\lambda_k=\lambda_0/(k+1)^b$, thus
\begin{align*}
& f(\bar x_N) - f^*  \\& \leq   \left(\sum_{k=0}^{N-1} \frac{\gamma_0^r}{(k+1)^{ar}}\right)^{-1} \left( m \left(C_f^2 + \lambda_0^2 C_h^2\right) \sum_{k=0}^{N-1} \frac{\gamma_0^{r+1}}{(k+1)^{a(r+1)}}  \right. \cr\\& \left.  
+m^2 C_f\left( C_f + \lambda_0 C_h \right) \sum_{k=0}^{N-1} \frac{\gamma_0^{r+1}}{(k+1)^{a(r+1)}} + 2 M_h \sum_{k=0}^{N-1} \frac{\lambda_0 \gamma_0^r}{(k+1)^{ar+b}}  \right. \cr\\& \left. + 2M^2\gamma_0^{r-1}N^{a(1-r)} \right). 
\end{align*}
Rearranging the terms, we have
\fy{\begin{align*}
&f(\bar x_N) - f^* \leq   \left(\sum_{k=0}^{N-1} \frac{\gamma_0^r}{(k+1)^{ar}}\right)^{-1} \\& \times\left(  2 M_h \sum_{k=0}^{N-1} \frac{\lambda_0 \gamma_0^r}{(k+1)^{ar+b}} + 2M^2\gamma_0^{r-1}N^{a(1-r)}  \right. \cr\\& \left. 
+\left( m^2 C_f\left( C_f + \lambda_0 C_h \right)+ m \left(C_f^2 + \lambda_0^2 C_h^2\right) \right) \sum_{k=0}^{N-1} \frac{\gamma_0^{r+1}}{(k+1)^{a(r+1)}}  \right). 
\end{align*}}
Let us define
\begin{align*}
&\hbox{ Term1}= \left(\sum_{k=0}^{N-1} \frac{1}{(k+1)^{ar}}\right)^{-1} N^{a(1-r)},\\
&\hbox{Term2}=\left(\sum_{k=0}^{N-1} \frac{1}{(k+1)^{ar}}\right)^{-1}\left(\sum_{k=0}^{N-1}\frac{1}{(k+1)^{ar+b}}\right).
\end{align*}
We have
\begin{align*}
&\hbox{Term1} \leq\frac{N^{a(1-r)}}{\frac{N^{1-ar}-1}{1-ar}}=\frac{(1-ar)N^{a(1-r)}}{N^{1-ar}-1}={\cal O}\left(N^{-(1-a)}\right), \\
&\hbox{Term2} \leq\frac{\frac{(N+1)^{1-ar-b}-1}{1-ar-b}+1}{\frac{N^{1-ar}-1}{1-ar}} =\frac{(1-ar)\left((N+1)^{1-ar-b}-1\right)}{(1-ar-b)\left(N^{1-ar}-1\right)} \\& + \frac{1-ar}{N^{1-ar}-1}  ={\cal O}\left(N^{-b}\right) + {\cal O}\left(N^{-(1-ar)}\right).
\end{align*}
So, we have
\begin{align*}
f(\bar x_N)-f^*\leq {\cal O}\left(N^{-\min \{1-ar,1-a,b\}}\right)= {\cal O}\left(N^{-\min \{1-a,b\}}\right),
\end{align*}
where we used $1-a\leq 1-ar$. Replacing $a$ and $b$ by their values, we have
\begin{align*}
f(\bar x_N)-f^*\leq {\cal O}\left(N^{-\min \{0.5-0.5\epsilon,0.5-\epsilon\}}\right)={\cal O}\left(N^{-(0.5-\epsilon)}\right).
\end{align*}
\end{proof}
\begin{table}[h]
	\setlength{\tabcolsep}{3pt}
	\centering
	\begin{tabular}{c| c  c  c}
	& $x_0=-10\times\mathbf{1}_n$ & $x_0=\mathbf{0}_n$ & $x_0=10\times\mathbf{1}_n$ \\ \hline\\
		\begin{turn}{90}
			$(10,1)$
		\end{turn}
		&
		\begin{minipage}{.14\textwidth}
			\includegraphics[scale=.17, angle=0]{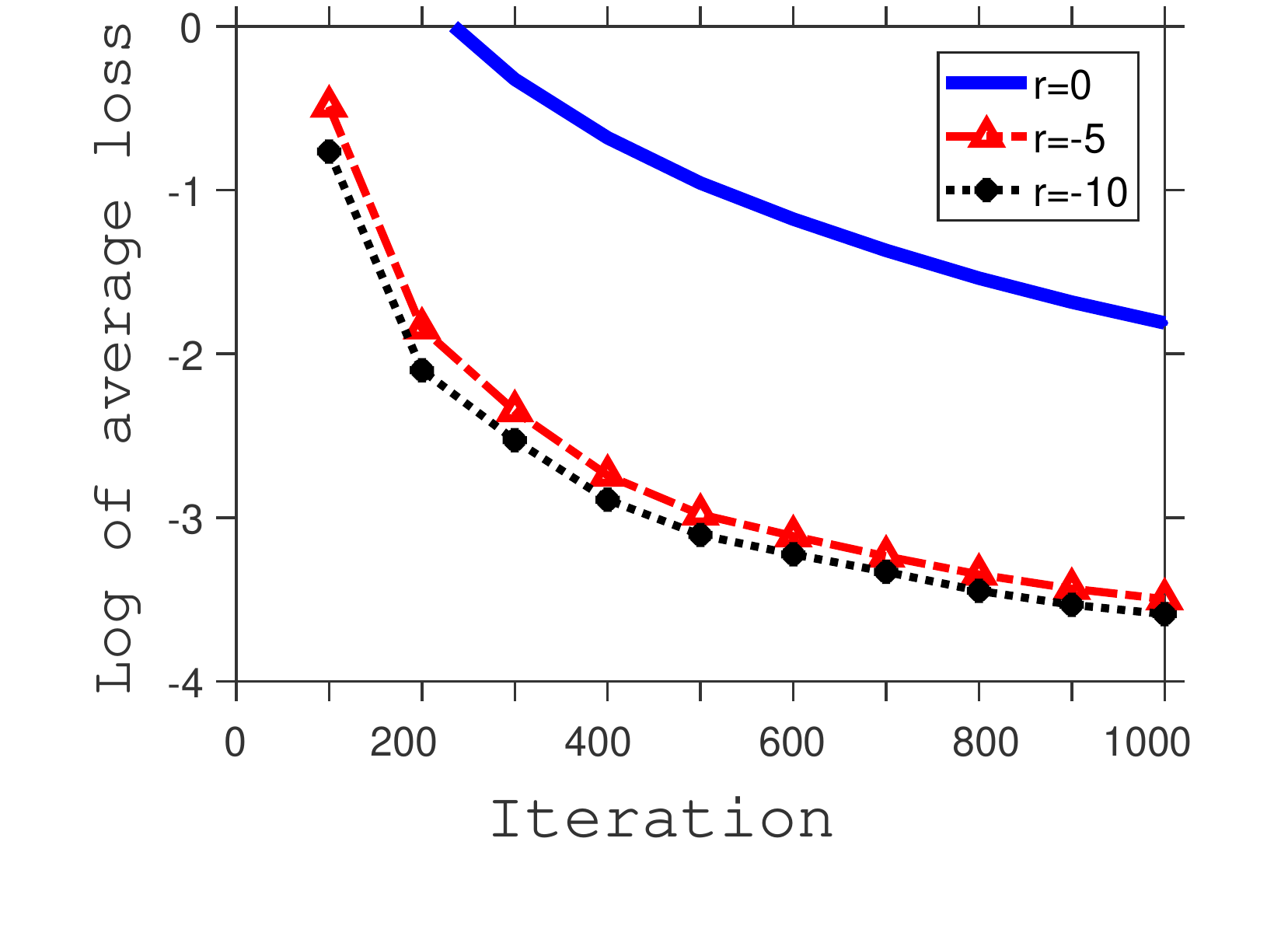}
		\end{minipage}
		&
		\begin{minipage}{.14\textwidth}
			\includegraphics[scale=.17, angle=0]{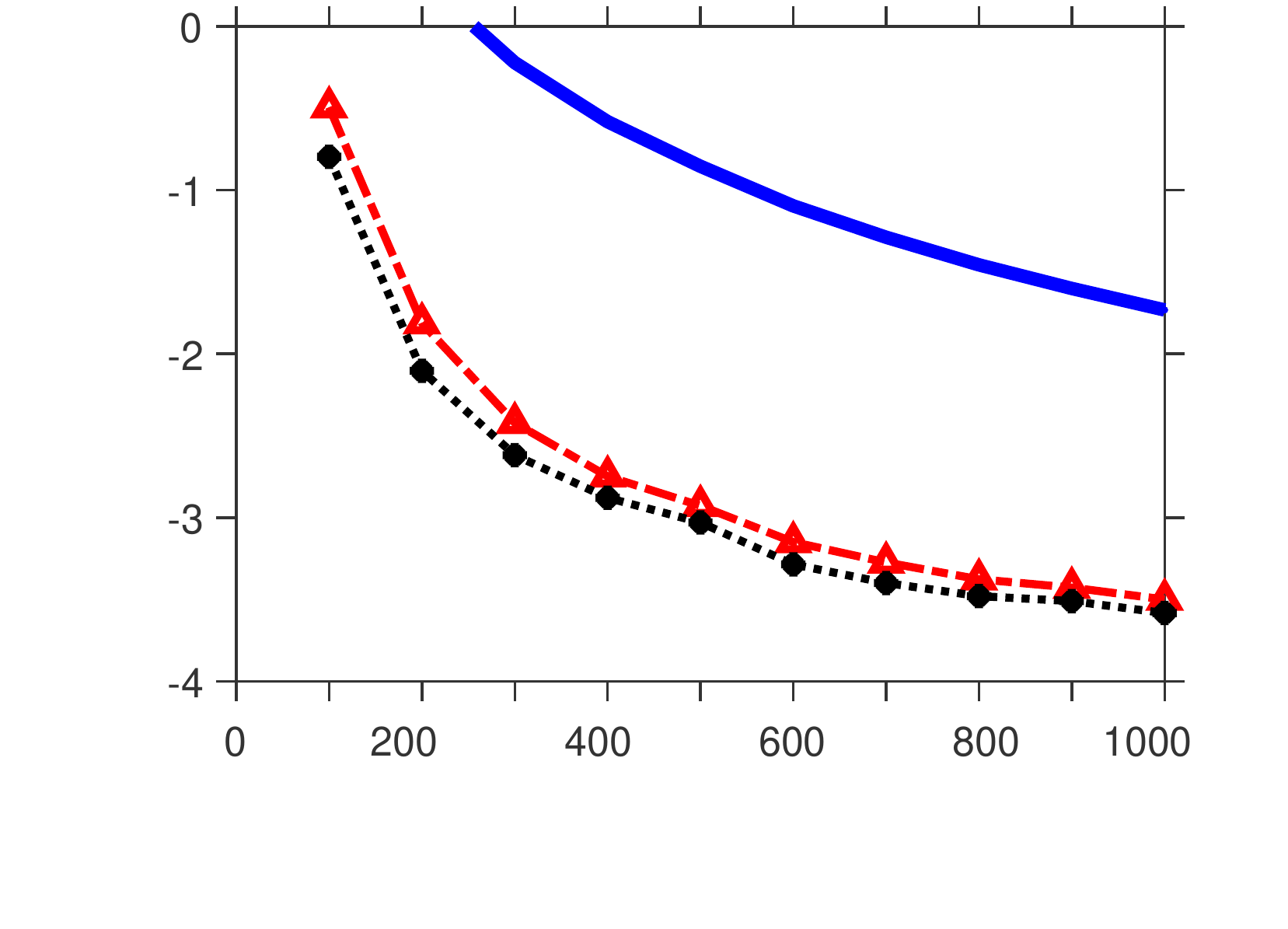}
		\end{minipage}
		&
		\begin{minipage}{.14\textwidth}
			\includegraphics[scale=.17, angle=0]{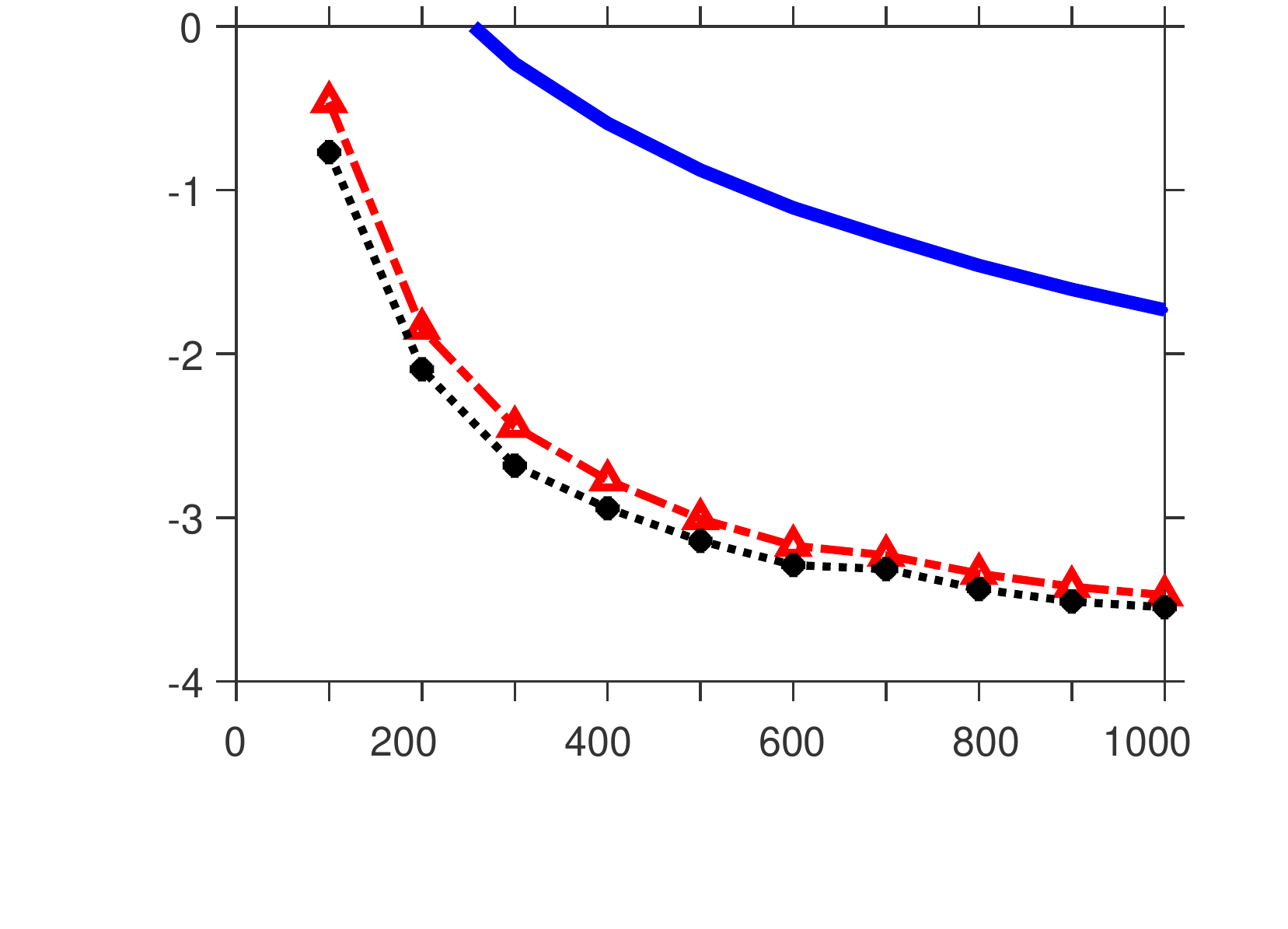}
		\end{minipage}
		\\
		\begin{turn}{90}
			$(1,10)$
		\end{turn}
		&
		\begin{minipage}{.14\textwidth}
			\includegraphics[scale=.17, angle=0]{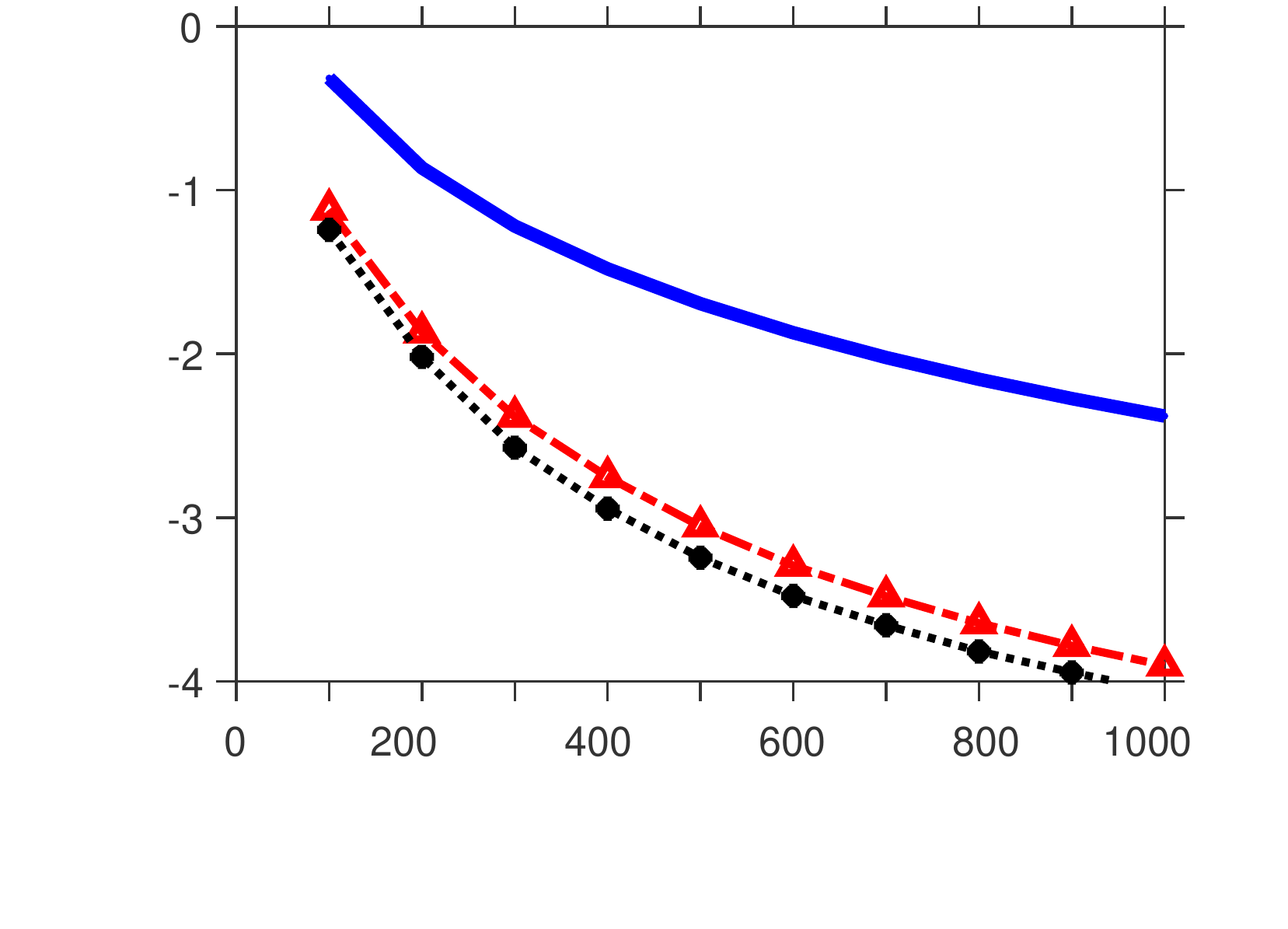}
		\end{minipage}
		&
		\begin{minipage}{.14\textwidth}
			\includegraphics[scale=.17, angle=0]{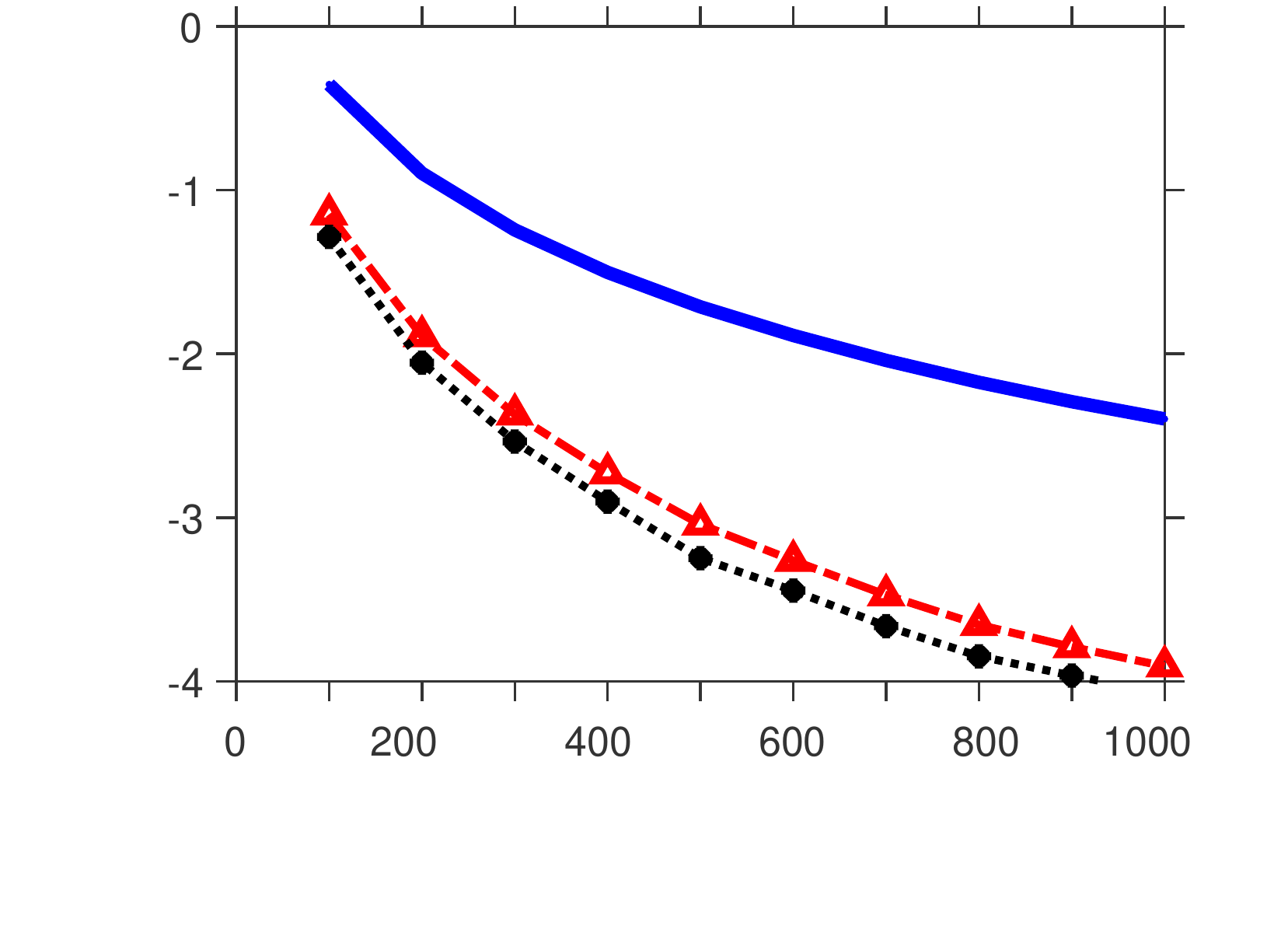}
		\end{minipage}
		&
		\begin{minipage}{.14\textwidth}
			\includegraphics[scale=.17, angle=0]{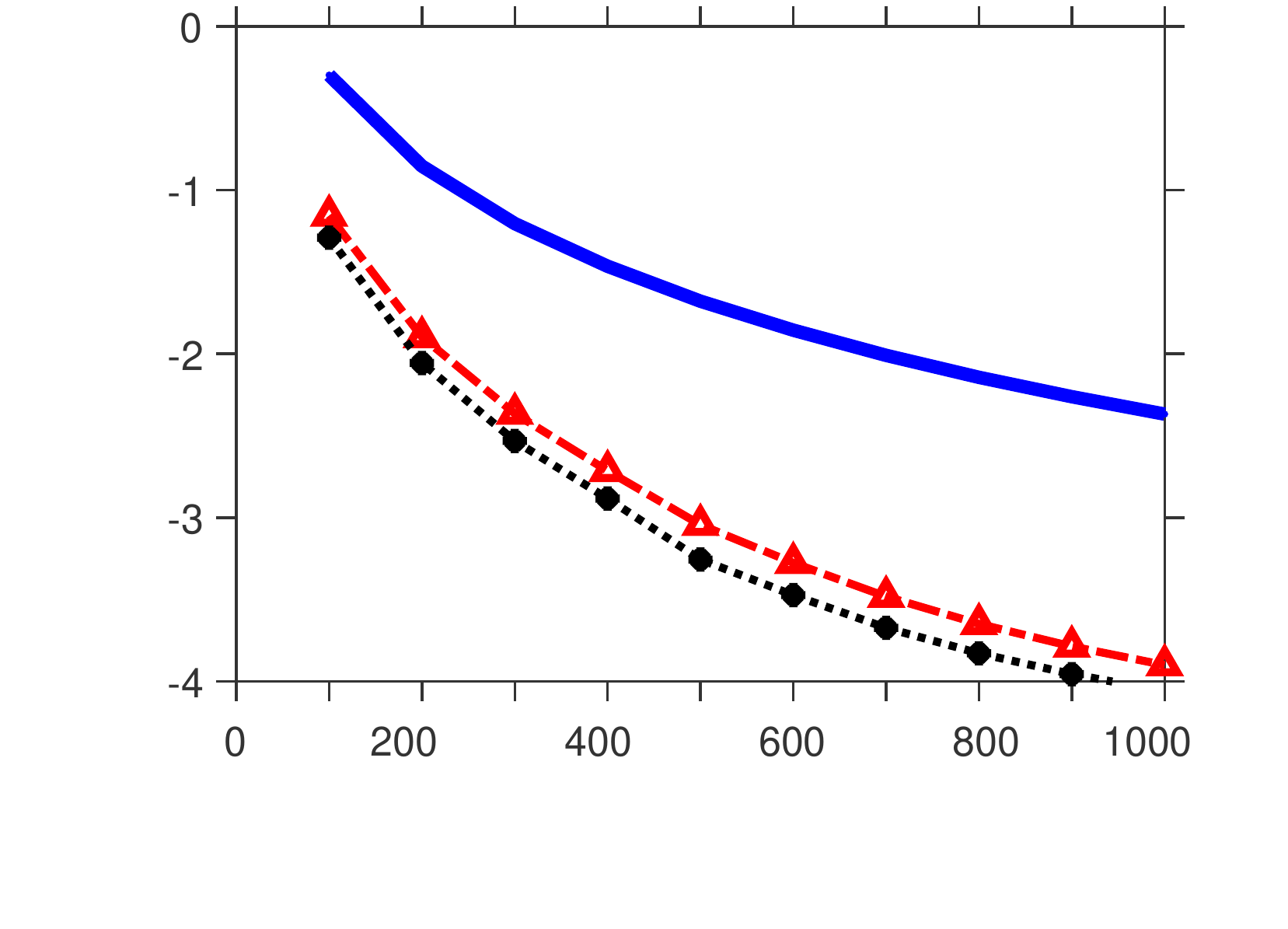}
		\end{minipage}
		\\
		\begin{turn}{90}
			$(0.1,0.1)$
		\end{turn}
		&
		\begin{minipage}{.14\textwidth}
			\includegraphics[scale=.17, angle=0]{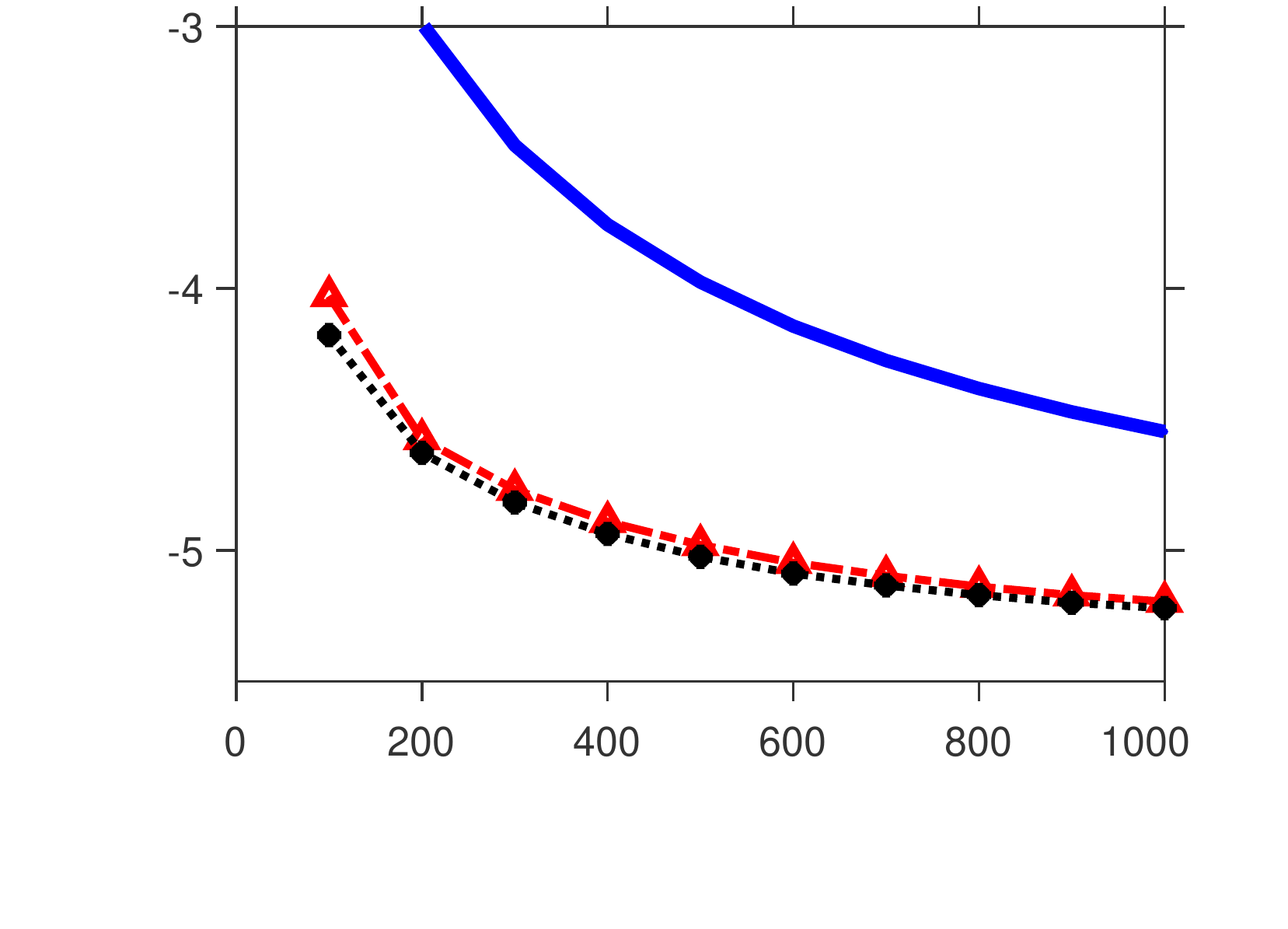}
		\end{minipage}
		&
		\begin{minipage}{.14\textwidth}
			\includegraphics[scale=.17, angle=0]{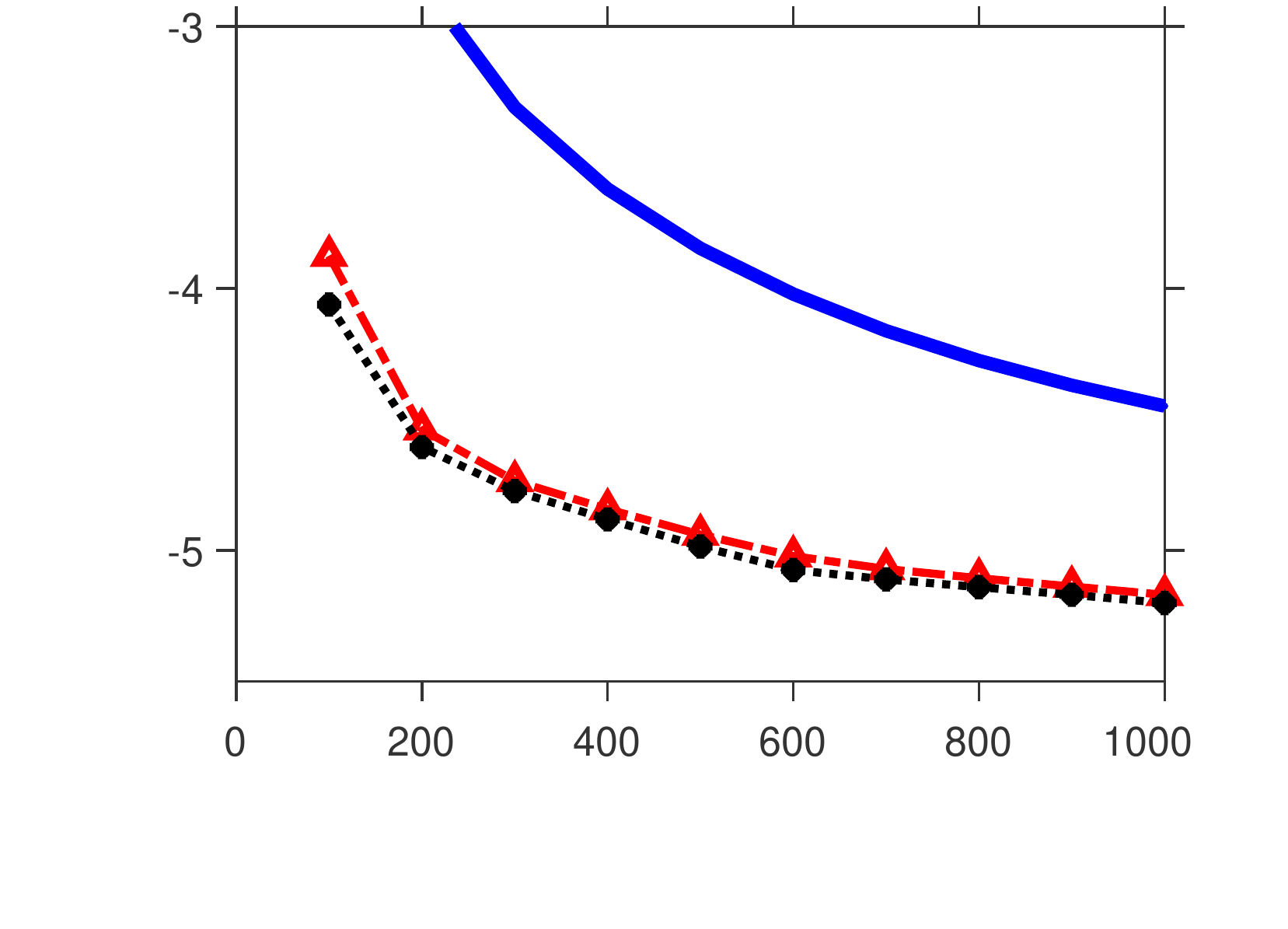}
		\end{minipage}
		&
		\begin{minipage}{.14\textwidth}
			\includegraphics[scale=.17, angle=0]{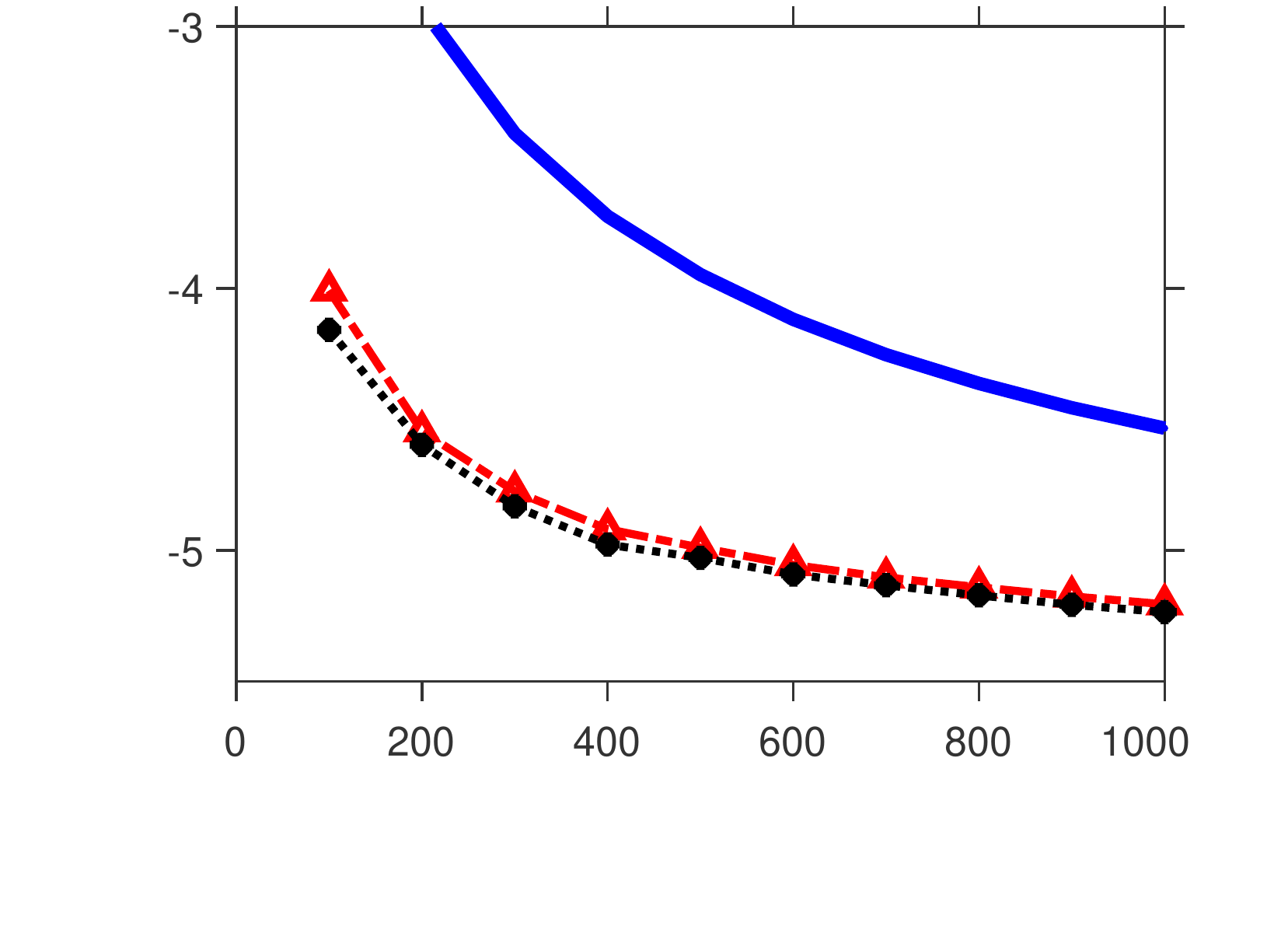}
		\end{minipage}
	\end{tabular}
	\captionof{figure}{\scriptsize{IR-IG performane for binary text classification problem with respect to different initial point $x_0$, parameters $(\gamma_0,\lambda_0)$ and $r$. The vertical axis represents the log of average loss and the horizontal axis is the number of iterations.}}
	\label{fig:fiveplots}
\end{table}
\section{Numerical results}\label{sec:num}
In this section, we apply the IR-IG method on a text classification problem. In this problem we assume to have a summation of hinge loss function, i.e., $\mathcal{L}(\langle x,a\rangle, b) \triangleq \max \{0, 1-b \langle x,a\rangle\}$ for any sample $a,b$ in the lower level of \eqref{def:SL-INC}. The set of observations $(a_i,b_i)$ are derived from the Reuters Corpus Volume I (RCV1) dataset  (see \cite{Lewis04}). This dataset has categorized Reuters articles, from 1996 to 1997, into four groups: \textit{ Corporate/Industrial, Economics, Government/Social} and \textit{Markets}. In this application, we use a subset of this dataset with $N=50,000$ articles and 138,921 tokens to perform a binary classification of articles only with respect to the \textit{Markets} class. Each vector $a_i$ represents the existence of all tokens in article $i$, and $b_i$ shows whether the article belongs to the \textit{Markets} class. To decide that a new article can be placed in the \textit{Markets} class, the problem in the lower level should be solved such that the optimal solution is a weight vector of the tokens regarding the \textit{Markets} class which minimizes the total loss. However, to make such a decision, an optimal solution with a large number of nonzero components are undesirable.  \fy{To induce sparsity,} we consider the \fy{following} bilevel problem:

\begin{align} \label{numeric upper2}
\displaystyle \mbox{minimize}& \ h(x)\triangleq  \frac{\mu_h }{2}\|x\|_2^2+ \|x\|_1\\
\mbox{subject to} & \ x \in \arg\min_{y \in \mathbb{R}^n } \sum_{i=1}^{m}\sum_{j=1}^{N/m}{ \mathcal{L}(a_{(i-1)N/m+j}^Ty, b_{(i-1)N/m+j})}, \nonumber
\end{align}
where, we let $\mu_h=0.1$ and we consider each batch of $N/m=1000$ articles to be one component function with the total number of component functions $m=50$. The function $h$ is strongly convex with parameter $\mu_h$. \fy{We} let $\gamma_k$ and $\lambda_k$ be given by the rules in Theorem \ref{thm rate}. We \fy{study the sensitivity of the method by changing} $x_0$, $\gamma_0$, $\lambda_0$, and the averaging parameter $r<1$. We finally report the logarithm of average of the loss function $\mathcal{L}$. The plots in Fig. 1 show the convergence of the IR-IG method for the problem \eqref{numeric upper2}. The results show the convergece of Algorithm \ref{algorithm:IRIG} with different initial values such as the starting point or parameters $\gamma_0, \lambda_0$ while when we pick a smaller $r$ the algorithm is faster in all the cases.

\section{Concluding remarks}\label{sec:rem}


Motivated by the applications of incremental \fy{gradient schemes} in distributed optimization, especially in machine learning and large data training, we develop an iterative regularized incremental first-order method,  called IR-IG, for solving a class of bilevel convex optimization. We prove the convergence of IR-IG and establish the corresponding rate in terms of the \fy{lower} level objective function. We finally \fy{apply IR-IG to} a binary text classification problem \fy{and} demonstrate the \fy{performance} of \fy{the proposed algorithm}.

\addtolength{\textheight}{-12cm}   







\begin{thebibliography}{99}
	
\bibitem{Beck17} A. Beck, First-order methods in optimization, MOS-SIAM Series on Optimization, Society
for Industrial and Applied Mathematics (SIAM), Philadelphia, PA, 2017.

\bibitem{Beck14} A. Beck and S. Sabach, A first-order method for finding minimal norm-like solutions of convex optimization problems, Mathematical Programming, 147(2) (2014), 25-46.  

\bibitem{Beck09} A. Beck and M. Teboulle, A fast iterative shrinkage-thresholding algorithm for linear inverse problems, SIAM journal on imaging sciences,  2(1) (2009), 183-202.

\bibitem{Bertsekas99}D. Bertsekas, Nonlinear programming, Athena Scientific, Belmont, MA, 1999.
		
\bibitem{Blatt07} D. Blatt, A. Hero, and H. Gauchman, A convergent incremental gradient method with a constant step-size, SIAM Journal Of Optimization,  18 (2007), 29-51.

\bibitem{Bottou05} L. Bottou and Y. Le Cun, On-line learning for very large data sets, Applied Stochastic Models in Business and Industry,  21 (2005), 137-151.

\bibitem{Defazio14} A. Defazio, F. Bach, and S. Lacoste-Julien, SAGA: A fast incremental gradient method with support for non-strongly convex composite objectives, in Advances in Neural Information	Processing Systems, MIT Press, Cambridge, MA, 2014, 1646-1654.

\bibitem{Tseng07} M. P. Friedlander and P. Tseng, Exact regularization of convex programs, SIAM Journal Of Optimization, 18(4) (2007), 1326-1350. 	
	
\bibitem{Rosasco18} G. Garrigos, L. Rosasco, and S. Villa, Iterative regularization via dual diagonal descent, Journal of Mathematical Imaging and Vision 60(2) (2018), 189-215.
	
\bibitem{Ozdaglar17} M. G{\"u}rb{\"u}zbalaban, A. Ozdaglar, and P. A. Parillo, On the convergence rate of incremental aggregated gradient algorithms, SIAM Journal Of Optimization,  27(2) (2018), 640-660.
	
\bibitem{Mangasarian79} S. -P. Han and O. L. Mangasarian, Exact penalty functions in nonlinear programming, 17(1) (1979), 251-269.

\bibitem{Knopp51} K. Knopp, Theory and applications of infinite series, Blackie \& Son Ltd., Glasgow, Great Britain, 1951.
				
\bibitem{Kocvara04} M. Ko\v{c}vara and J. V. Outrata, Optimization problems with equilibrium constraints and their numerical solution, Mathematical Programming,  101(1) (2004), 119-149.	
	
\bibitem{Lewis04} D. D. Lewis, Y. Yang, T. G. Rose, and F. Li, RCV1: A new benchmark collection for text categorization research, Journal of Machine Learning Research, 5 (2004), 361-397.	

\bibitem{Pang96} Z.-Q. Luo, J.-S. Pang and D. Ralph, Mathematical programs with equilibrium constraints, Cambridge University Press, Cambridge (1996).	
	
\bibitem{Mairal13} J. Mairal, Optimization with first-order surrogate functions, in ICML, JMLR Proceedings 28, 2013, 783-791.	

\bibitem{Mangasarian85} O. L. Mangasarian, Sufficiency Of exact penalty minimization, SIAM Journal on Control and Optimization,  23(1) (1985), 30-37. 
	
\bibitem{Nedich01} A. Nedi\'c and D. Bertsekas, Incremental subgradient methods for nondifferentiable optimization, SIAM Journal of Optimization,  12(1) (2001), 109-138.	
	
\bibitem{Neto11} E. S. H. Neto and \'A. A. R. De Pierro, On perturbed steepest descent methods with inexact	line search for bilevel convex optimization, Optimization, 60 (2011), 991-1008.	

\bibitem{Polyak87} B. T. Polyak, Introduction to optimization, Optimization Software, Inc., New York, 1987. 

\bibitem{Roux12} N. L. Roux, M. Schmidt, and F. R. Bach, A stochastic gradient method with an exponential convergence rate for finite training sets, in Advances in Neural Information Processing Systems 25, F. Pereira, C. J. C. Burges, L. Bottou, and K. Q. Weinberger, eds., MIT Press, Cambridge, MA, 2012, 2663-2671.

\bibitem{Sabach17} S. Sabach and S. Shtern, A first-order method for solving convex bilevel optimization Problems, SIAM Journal on Optimization, 27(2) (2017), 640-660.

\bibitem{Solodov072} M. Solodov, A bundle method for a class of bilevel nonsmooth convex minimization problems, SIAM Journal Of Optimization,  18(1) (2007), 242-259.

\bibitem{Solodov07} M. Solodov,  An explicit descent method for bilevel convex optimization, Journal of Convex Analysis,   14(2) (2007).

\bibitem{Tikhonov77} A. N. Tikhonov and V. Y. Arsenin, Solutions of ill-posed problems, V. H. Winston and Sons, Washington, D. C., 1977. Translated from Russian.	

\bibitem{Tseng14} P. Tseng and S. Yun, Incrementally updated gradient methods for constrained and regularized	optimization, Journal of Optimization Theory and Applications,  160 (2014), 832-853.


\bibitem{Xu04} H.-K. Xu, Viscosity approximation methods for nonexpansive mappings, Journal of Mathematical Analysis and Applications, 298(1) (2004), 279-291.
		
\bibitem{Yamada11} I. Yamada, M. Yukawa, and M. Yamagishi, Minimizing the \fy{M}oreau envelope of nonsmooth	convex functions over the fixed point set of certain quasi-nonexpansive mappings, in Fixed-Point Algorithms for Inverse Problems in Science and Engineering, Springer, New York, 2011, 345-390.
										
\bibitem{Yousefian17} F. Yousefian, A. Nedi\'c, and U. V. Shanbhag., On Smoothing, regularization and averaging in stochastic approximation methods for stochastic variational inequality problems, Mathematical Programming, 165 (1) (2017), 391-431.

\bibitem{Yousefian18} F. Yousefian, A. Nedi\'c, and U. V. Shanbhag., On stochastic mirror-prox algorithms for stochastic \fy{C}artesian variational inequalities: randomized block coordinate, and optimal averaging schemes, Set-Valued and Variational Analysis, (2018), https://doi.org/10.1007/s11228-018-0472-9.








	
\end{thebibliography}
\end{document}